\documentclass[a4paper]{amsart}

\usepackage{amscd,amsthm,amsfonts,amssymb,amsmath}
\usepackage[colorlinks=true]{hyperref}
\usepackage{fullpage}
\usepackage[all]{xy}

    \newcommand{\BA}{{\mathbb {A}}} \newcommand{\BB}{{\mathbb {B}}}
    \newcommand{\BC}{{\mathbb {C}}} 
     
     \newcommand{\BH}{{\mathbb {H}}}

    \newcommand{\BQ}{{\mathbb {Q}}} \newcommand{\BR}{{\mathbb {R}}}

     \newcommand{\BZ}{{\mathbb {Z}}}

     \renewcommand{\CD}{{\mathcal {D}}}
    \newcommand{\CE}{{\mathcal {E}}} 
     \newcommand{\CH}{{\mathcal {H}}}
     
    \newcommand{\CK}{{\mathcal {K}}} \newcommand{\CL}{{\mathcal {L}}}
     \newcommand{\CN}{{\mathcal {N}}}
    \newcommand{\CO}{{\mathcal {O}}} \newcommand{\CP}{{\mathcal {P}}}
     \newcommand{\CR}{{\mathcal {R}}}
    \newcommand{\CS}{{\mathcal {S}}} 
     
    \newcommand{\CW}{{\mathcal {W}}} \newcommand{\CX}{{\mathcal {X}}}

     \newcommand{\RN}{{\mathrm {N}}}

    \newcommand{\fa}{{\mathfrak{a}}}

    \newcommand{\fg}{{\mathfrak{g}}}

     \newcommand{\fn}{{\mathfrak{n}}}
     \newcommand{\fp}{{\mathfrak{p}}}

    \newcommand{\ab}{{\mathrm{ab}}}
    \newcommand{\ad}{{\mathrm{ad}}}
    
    \newcommand{\Aut}{{\mathrm{Aut}}}

    \newcommand{\End}{{\mathrm{End}}}

    \newcommand{\Gal}{{\mathrm{Gal}}} \newcommand{\GL}{{\mathrm{GL}}}
    
    \newcommand{\Hom}{{\mathrm{Hom}}}
    
    \newcommand{\id}{{\mathrm{id}}}
    \newcommand{\Ind}{{\mathrm{Ind}}}

     \newcommand{\JL}{{\mathrm{JL}}}
    
    \newcommand{\length}{{\mathrm{length}}}
    \newcommand{\Lie}{{\mathrm{Lie}}}
    \newcommand{\Mass}{{\mathrm{Mass}}}

    \newcommand{\ord}{{\mathrm{ord}}} \newcommand{\rank}{{\mathrm{rank}}}
     \newcommand{\Pic}{\mathrm{Pic}}
    \newcommand{\pr}{{\mathrm{pr}}} 
    
    \renewcommand{\mod}{\ \mathrm{mod}\ }\renewcommand{\Re}{{\mathrm{Re}}}

    \newcommand{\Res}{{\mathrm{Res}}}

    \newcommand{\SL}{{\mathrm{SL}}}
     
    \newcommand{\SO}{{\mathrm{SO}}}
    \newcommand{\SU}{{\mathrm{SU}}}
    \newcommand{\sgn}{{\mathrm{sgn}}}

    \newcommand{\tr}{{\mathrm{tr}}}\newcommand{\tor}{{\mathrm{tor}}}
    \newcommand{\RTr}{{\mathrm{Tr}}}
    
    \newcommand{\Vol}{{\mathrm{Vol}}}
    
    \newcommand{\Pet}{{\mathrm{Pet}}}

\newcommand{\matrixx}[4]{\begin{pmatrix}
#1 & #2 \\ #3 & #4
\end{pmatrix} }        

    \font\cyr=wncyr10

    \newcommand{\Sha}{\hbox{\cyr X}}\newcommand{\wt}{\widetilde}
    \newcommand{\wh}{\widehat}
    
    \newcommand{\pair}[1]{\langle {#1} \rangle}

    \newcommand{\ov}{\overline}

    \newcommand{\sk}{\medskip}
    \newcommand{\lra}{\longrightarrow}
    \newcommand{\ra}{\rightarrow} 
    \newcommand{\lto}{\longmapsto}\newcommand{\bs}{\backslash}
    
    \newcommand{\s}{\sk\noindent}

    \theoremstyle{plain}
    \newtheorem{thm}{Theorem}[section] 
    \newtheorem{lem}[thm]{Lemma}  \newtheorem{prop}[thm]{Proposition}
     \newtheorem{defn}[thm]{Definition}

\theoremstyle{remark} 
\theoremstyle{remark} 
\theoremstyle{remark} 

    \newcommand{\Neron}{N\'{e}ron~}\newcommand{\adele}{ad\'{e}le~}
    
    \newcommand{\adelic}{ad\'{e}lic}

    \numberwithin{equation}{section}

\begin{document}

\title{Explicit Gross-Zagier and Waldspurger Formulae}

\author{Li Cai, Jie Shu,  and Ye Tian}

\address{Li Cai: Mathematical Sciences Center, Tsinghua University, Beijing
100084} \email{lcai@math.tsinghua.edu.cn}

\address{Jie Shu: Academy of Mathematics and Systems
Science, Morningside center of Mathematics, Chinese Academy of
Sciences, Beijing 100190} \email{shujie09@mails.gucas.ac.cn}

\address{Ye Tian: Academy of Mathematics and Systems
Science, Morningside center of Mathematics, Chinese Academy of
Sciences, Beijing 100190} \email{ytian@math.ac.cn}
\thanks{Ye Tian was supported  by grants  NSFC  11325106 and NSFC 11321101.}

\maketitle

\tableofcontents

\section{Main Results}
\subsection{Introduction}

The Gross-Zagier formula and the Waldspurger formula are probably the two most important analytic tools known at present for studying the still largely unproven conjecture of Birch and Swinnerton-Dyer.
Much work has already been done on both formulae. In particular, the recent book by Yuan-Zhang-Zhang \cite{YZZ} establishes what is probably the most general
case of the Gross-Zagier formula. Nevertheless, when it comes to actual applications to the arithmetic of elliptic curves or abelian varieties, one very often needs a
more explicit form of the Gross-Zagier formula than that given in \cite{YZZ}, and similarly a more explicit form of the Waldspurger formula than one finds in the existing literature.  This is clearly illustrated, for example, by the papers \cite{BD},
\cite{Tian2}, \cite{TYZ}, \cite{CLTZ}. Our aim here is to establish what we believe are the most general explicit versions of both formulae,
namely Theorems \ref{GZ} and \ref{variation1} for the Gross-Zagier formula,
and Theorems \ref{W} and \ref{variation2} for the Waldspurger formula. Our methods have been directly inspired by \cite{YZZ},
and also the ideas of Gross \cite{Gross} and Gross-Prasad \cite{GP}.

\medskip

In the remainder of this introduction, we would like to explain in detail our explicit formulae in the simplest, and most important, case of modular forms over $\BQ$. Let $\phi$ be a newform of weight $2$, level $\Gamma_0(N)$, with Fourier expansion $\phi=\sum_{n=1}^\infty a_n q^n$ normalized such that $a_1=1$. Let $K$ be an imaginary quadratic field of discriminant $D$ and $\chi$ a primitive ring class character over $K$ of conductor $c$, i.e. a character of $\Pic(\CO_c)$ where $\CO_c$ is the order $\BZ+c\CO_K$ of $K$.  Assume the Heegner condition (first introduced by Birch in a special case):-
\begin{enumerate}
\item  $(c, N)=1$,  and no prime divisor $p$ of $N$ is  inert in $K$,  and  also $p$ must be split in $K$ if $p^2|N$.
\item  $\chi([\fp])\neq a_p$ for any prime $p|(N, D)$, where $\fp$ is the unique prime ideal of $\CO_K$ above $p$  and $[\fp]$  is its class in $\Pic(\CO_c)$.
\end{enumerate}
 Let $L(s, \phi, \chi)$ be the Rankin L-series of $\phi$ and the theta series $\phi_\chi$ associated to $\chi$ (without the local Euler factor at infinity). It follows from the Heegner condition that the sign in the functional  equation of $L(s, \phi, \chi)$ is $-1$. Let $(\phi,\phi)_{\Gamma_0(N)}$ denote the Petersson norm of $\phi$:
$$(\phi, \phi)_{\Gamma_0(N)}=\iint_{\Gamma_0(N)\bs \CH} |\phi(z)|^2 dx dy, \qquad z=x+iy.$$
\noindent Let $X_0(N)$ be the modular curve over $\BQ$, whose $\BC$-points parametrize isogenies $E_1\ra E_2$ between elliptic curves over $\BC$ whose kernel  is cyclic of order $N$. By  the Heegner condition, there exists a proper ideal $\CN$ of $\CO_c$ such that $\CO_c/\CN \cong \BZ/N\BZ$. For any proper ideal $\fa$ of $\CO_c$, let $P_\fa\in X_0(N)$ be the point representing the isogeny $\BC/\fa \ra \BC/\fa\CN^{-1}$, which is defined over the ring class field $H_c$ over $K$ of conductor $c$, and only depends on the class of $\fa$ in $\Pic(\CO_c)$.  Let $J_0(N)$ be the Jacobian of $X_0(N)$.  Writing $\infty$ for the cusp at infinity on $X_0(N)$, we have the morphism from $X_0(N)$ to $J_0(N)$ over $\BQ$ given by $P\mapsto [P-\infty]$. Let $P_\chi$  be the point $$P_\chi=\sum_{[\fa]\in \Pic(\CO_c)} [P_\fa-\infty]\otimes \chi([\fa])\in J_0(N)(H_c)\otimes_\BZ \BC, $$ and write $P_\chi^\phi$  for the  $\phi$-isotypical component of $P_\chi$.

\medskip

 The following theorem was proved by Gross-Zagier in the case $c=1$ in their celebrated work \cite{GZ}, and follows immediately from the general explicit Gross-Zagier formula in Theorem \ref{GZ} (see the special case 2,  and the example after Theorem \ref{GZ}.)
\begin{thm} \label{A} Let $\phi, \chi$ be as above satisfying the Heegner conditions (1) and (2).  Then
$$L'(1, \phi, \chi)=2^{-\mu(N, D)} \cdot \frac{8\pi^2 (\phi, \phi)_{\Gamma_0(N)}}{u^2 \sqrt{|Dc^2|}} \cdot \wh{h}_K(P_\chi^{\phi}),$$
where $\mu(N, D)$ is the number of prime factors of the greatest common divisor of $N$ and $D$, $u=[\CO_c^\times:\BZ^\times]$ is half of the number of roots of unity in $\CO_c$, and $\wh{h}_K$ is the \Neron-Tate height on $J_0(N)$ over $K$.
\noindent In particular, if $\phi$ is associated to an elliptic curve $E$ over $\BQ$  via Eichler-Shimura theory and $f: X_0(N)\ra E$ is a modular parametrization mapping the cusp $\infty$ to the identity $O\in E$, then the Heegner divisor $P^0_\chi(f):=\sum_{[\fa] \in \Pic(\CO_c)} f(P_\fa)\otimes \chi([\fa])\in E(H_c)_\BC$ satisfies
$$L'(1, E, \chi)=2^{-\mu(N, D)} \cdot \frac{8\pi^2 (\phi, \phi)_{\Gamma_0(N)}}{u^2 \sqrt{|Dc^2|}} \cdot \frac{\wh{h}_K(P^0_\chi(f))}{\deg f},$$
where $\wh{h}_K$ is the \Neron-Tate height on $E$ over $K$ and $\deg f$ is the degree of the morphism $f$.
\end{thm}

\medskip

 Comparing the above Gross-Zagier formula with the conjecture of Birch and Swinnerton-Dyer for $L(E/K, s)$, we immediately
 obtain the following:

\s{\bf Conjecture}. {\em Let $E$ be an elliptic curve defined over $\BQ$ of conductor $N$,  and let $K$ an imaginary quadratic field of discriminant $D$ such that for any prime $\ell$ dividing $N$, either $\ell$ splits in $K$, or $\ell$ is ramified in $K$ and $\ell^2$ exactly divides $N$.
Let $f: X_0(N)\ra E$ be a modular parametrization mapping $\infty$ to $O$. Let $\CN\subset \CO_K$ be any ideal with $\CO_K/\CN\cong \BZ/N\BZ$,  let $P\in X_0(N)(H_K)$ be the point representing the isogeny $(\BC/\CO_K\ra \BC/\CN^{-1})$, and write $P_K(f):=\RTr_{H_K/K} f(P)\in E(K)$. Assume $P_K(f)$  is not torsion. Then
$$\sqrt{\#\Sha(E/K)}= 2^{-\mu(N, D)}\cdot \frac{[E(K): \BZ P_K(f)]}{C\cdot [\CO_K^\times:\BZ^\times]\cdot \prod_{\ell|\frac{N}{(N, D)}} m_\ell},$$
where $m_\ell=[E(\BQ_\ell): E^0(\BQ_\ell)]$,  and $C$ is the positive integer such that if $\omega_0$ is a \Neron differential on $E$ then
$f^*\omega_0=\pm C\cdot 2\pi i \phi(z)dz$.}
\bigskip

We next state our explicit Waldspurger formula over $\BQ$. Let $\phi=\sum_{n=1}^\infty a_n q^n \in S_2(\Gamma_0(N))$ be a newform of weight $2$ and level $\Gamma_0(N)$. Let $K$ be an imaginary  quadratic field and $\chi: \Gal(H_c/K)\lra \BC^\times$ a character of conductor $c$. Assume the following  conditions:
\begin{enumerate}
\item[(i)] $(c, N)=1$ and if $p|(N, D)$, then $p^2\nmid N$;
\item[(ii)] let $S$ be the set of places $p|N\infty$ non-split in $K$ such that for a finite prime $p$,  $\ord_p(N)$ is odd if $p$ is inert in $K$, and  $\chi([\fp])=a_p$ if $p$ is ramified in $K$. Then   $S$  has even cardinality.
\end{enumerate}
It follows that the sign of the functional equation of the Rankin L-series $L(s, \phi, \chi)$ is $+1$.
\noindent Let $B$ be the quaternion algebra over $\BQ$ ramified exactly at places in $S$. Note that the condition (ii) implies that there exists an embedding of $K$ into $B$ which we fix once and for all. Let $R\subset B$ be an order of discriminant $N$ with $R\cap K=\CO_c$. Such an order exists and is unique up to conjugation by $\wh{K}^\times$. Here, for an abelian group $M$, we define $\wh{M}=M\otimes_\BZ \wh{\BZ}$, where $\wh{\BZ}=\prod_p \BZ_p$, with $p$ running over all primes. By the reduction theory of definite quadratic forms, the coset $X:=B^\times \bs \wh{B}^\times/\wh{R}^\times$ is finite,   say of order $n$. Let $g_1, \cdots, g_n$ in $\wh{B}^\times$ represent  the distinct classes $[g_1], \cdots. [g_n]$. For each $i=1, \cdots, n$, let $\Gamma_i=(B^\times\cap g_i \wh{B}^\times g_i^{-1})/\{\pm 1\}$. Then $\Gamma_i$ is a finite group, and we denote its order by $w_i$.
Let $\BZ[X]$ denote the free $\BZ$-module of  formal sums $\sum_{i=1}^n a_i [g_i]$ with $a_i\in\BZ$, and define a height pairing on $\BZ[X]$ by
$$\pair{\sum a_i [g_i], \sum b_i [g_i]}=\sum_{i=1}^n a_i b_i w_i,$$
which is positive definite on $\BR[X]:=\BZ[X]\otimes_\BZ \BR$  and has a natural Hermitian extension to $\BC[X]:=\BZ[X]\otimes_\BZ \BC$. Define the degree of a vector $\sum a_i[g_i]\in \BZ[X]$ to be $\sum a_i$ and let $\BZ[X]^0$ denote the degree 0 submodule of $\BZ[X]$.
Then $\BZ[X]$ and $\BZ[X]^0$ are endowed with actions of Hecke operators $T_p, S_p, p\nmid N$ which are linear and defined as follows. For any prime $p\nmid N$, $B_p^\times/R_p^\times \cong \GL_2(\BQ_p)/\GL_2(\BZ_p)$ can be identified with the set of $\BZ_p$-lattices  in a $2$-dimensional vector space over $\BQ_p$. Then for any $g=(g_v)\in \wh{B}^\times$,
$$S_p([g])=[g^{(p)} s_p(g_p)], \qquad T_p([g])=\sum_{h_p} [g^{(p)} h_p],$$
where $g^{(p)}$ is the $p$-off part of $g$, namely $g^{(p)}=(g^{(p)}_v)$ with $g^{(p)}_v=g_v$ for all $v\neq p$ and $g^{(p)}_p=1$,  and if $g_p$ corresponds to lattice $\Lambda$, then $s_p(g_p)$ is the coset corresponding to the homothetic lattice $p\Lambda$, and $h_p$ runs over $p+1$ lattices $\Lambda'\subset \Lambda$ with $[\Lambda: \Lambda']=p$.  There is a unique line $V_\phi\subset  \BC[X]^0$ where $T_p$ acts as $a_p$ and $S_p$ acts trivially for all $p\nmid N$.   Recall that the fixed embedding of $K$ into $B$ induces a map
$$\Pic(\CO_c)=K^\times \bs \wh{K}^\times/ \wh{\CO}_c^\times \lra X=B^\times\bs \wh{B}^\times/\wh{R}^\times, \quad t\mapsto x_t,$$
using which we define an element in $\BC[X]$,
 $$P_\chi:=\sum \chi^{-1}(t) x_t$$
 and let $P_\chi^\phi$ be its projection to the line $V_\phi$.
\noindent The following explicit height formula for $P_\chi^\phi$, which was proved by Gross in some case in \cite{Gross1}, is a special case of the explicit Waldspurger formulas in Theorems \ref{W} and \ref{G} (with Proposition \ref{local-multi-one-2}).
\begin{thm}\label{B} Let $(\phi, \chi)$ be as above satisfying the conditions (i) and (ii). Then we have
$$L(1, \phi, \chi)=2^{-\mu(N, D)} \cdot \frac{8\pi^2(\phi, \phi)_{\Gamma_0(N)}}{u^2 \sqrt{|Dc^2|}} \cdot \pair{P_\chi^\phi, P_\chi^\phi},$$
where $\mu(N, D)$ is the number of prime factors of the greatest common divisor of $N$ and  $D$, $u=[\CO_c^\times:\BZ^\times]$ is half of number of roots of unity in $\CO_c$.  Let  $f=\sum_i f(g_i) w_i^{-1} [g_i]$ be any non-zero vector on the line $V_\phi$ and let $P^0_\chi(f)=\sum_{t\in \Pic(\CO_c)} f(t) \chi(t)$.  Then the above formula can be rewritten as
$$L(1, \phi, \chi)=2^{-\mu(N, D)} \cdot \frac{8\pi^2(\phi, \phi)_{\Gamma_0(N)}}{u^2 \sqrt{|Dc^2|}} \cdot \frac{|P^0_\chi(f)|^2}{\pair{f, f}}.$$
\end{thm}

\s{\bf Acknowledgements}. The authors thank J. Coates, H. Darmon,  B. Gross, D. Prasad, W. Xiong,  X. Yuan, S. Zhang, and W. Zhang for encouragement and helpful discussions.

\medskip

\s{\bf Notations for First Two Sections}. We denote by $F$ the base number field of degree $d=[F:\BQ]$ over $\BQ$ and $\CO=\CO_F$ its ring of integers with different $\delta$.  Let $\BA=F_\BA$ be the \adele ring of $F$ and $\BA_f$ its finite part. For any $\BZ$-module $M$, we denote by $\wh{M}=M\otimes_\BZ \wh{\BZ}$ and $\wh{\BZ}=\prod_p \BZ_p$. For example, $\wh{F}=\BA_f$. Let $|\ |_\BA: \BA^\times \lra \BR_+^\times$ denote the standard \adelic\ absolute value so that $d(ab)=|a|_\BA db$ for any Haar measure $db$ on $\BA$. Let $|\ |_v$ denote the absolute value on $F_v^\times$ for each place $v$ of $F$ such that $|x|_\BA=\prod_v |x_v|_v$ for any $x=(x_v)\in \BA^\times$. For any non-zero fractional ideal $b$ of $F$, let $\|b \|$ denote the norm of $b$.  For any $x\in \BA_f^\times$ we also write $\|x\|$ for $\|b_x\|$ where $b_x$ is the ideal corresponding to $x$ so that $\|x\|=|x|_\BA^{-1}$ and for any non-zero fractional ideal $b$ we also write $|b|_\BA$ for $|x_b|_\BA$ with any $x_b\in \BA_f^\times$ whose corresponding ideal is $b$ so that $|b|_\BA=\|b\|^{-1}$. For a finite place $v$, sometimes we also denote by $v$ its corresponding prime ideal and  $q_v=\# \CO/v$.  For a fractional ideal $b$ of $F$,  we write $|b|_v=|x_b|_v$ for $x_b\in F_v$ with $x_b\CO_v=b\CO_v$, denote by  $\ord_v(b)$ the additive valuation of $b$ at $v$ so that $\ord_v(v)=1$ and write $v\| b$ if $\ord_v(b)=1$. We denote by $\infty$  the set of infinite places of $F$.  Denote by $L(s, 1_F)$ the complete L-series for the trivial Hecke character $1_F$ on $\BA^\times$ so that $L(s, 1_F)=\Gamma_\BR(s)^{r_1}\Gamma_\BC(s)^{r_2} \zeta_F(s)$, where $r_1$ (resp $r_2$) is the number of real (resp. complex) places of $F$, $\zeta_F(s)$ is the usual Dedekind zeta-function of $F$, $\Gamma_\BR(s)=\pi^{-s/2}\Gamma(s/2)$, and $\Gamma_\BC(s)=2(2\pi)^{-s}\Gamma(s)$. For each place $v$ of $F$, let $L(s, 1_v)$ denote the local Euler factor of $L(s, 1_F)$ at $v$. Let $D_F$ denote the absolute discriminant of $F$ and $\delta \subset \CO$ the different of $F$ so that $\|\delta\|=|D_F|$.

In the first two sections, we denote by $K$ a quadratic extension over $F$, $D=D_{K/F}\subset \CO$ the relative discriminant of $K$ over $F$, and $D_K$ the absolute discriminant of $K$. Let $K^\ab$ denote the maximal abelian extension over $K$ and let $\sigma: K_\BA^\times/K^\times \ra \Gal(K^\ab/K)$ denote the Artin reciprocity map in the class field theory. For any non-zero ideal $b$ of $\CO$ let $\CO_b=\CO+b\CO_K$ be the unique $\CO$-order of $K$ satisfying $[\CO_K: \CO_b]=\# \CO/b$ and we call $b$ its conductor. For any finite place $v$ of $F$, $\CO_{b, v}=\CO_b\otimes_{\CO} \CO_v$ only depends on $\ord_v b$. Thus for a fractional ideal $b$  and a finite place $v$ of $F$, $\CO_{b, v}$ makes sense if $\ord_v b\geq 0$. Let $\Pic_{K/F}(\CO_b)=\wh{K}^\times/K^\times \wh{F}^\times \wh{\CO}_b^\times$. Then there is an exact sequence
$$\Pic(\CO_F)\ra \Pic (\CO_b) \ra \Pic_{K/F}(\CO_b)\ra 0.$$
Let $\kappa_b$ be the kernel of the first arrow, which has order $1$ or $2$   in the case $F$ is totally real and $K$ is a totally imaginary quadratic extension over $F$ (see Theorem 10.3 in \cite{Washington}).

For any algebraic group $G$ over $F$, let $G_\BA=G(\BA)$ be the group of \adelic\ points on $G$. For a finite set $S$ of places of $F$, let $G_S=\prod_{v\in S} G(F_v)$ (resp. $G_\BA^{(S)}=G(\BA)^{(S)}$) the $S$-part of $G_\BA$ (resp. the $S$-off part of $G_\BA$)  viewed as a subgroup of $G_\BA$ naturally so that the $S$-off components (resp, $S$-components) are constant $1$.  More general, for a subgroup $U$ of $G_\BA$ of form $U=U_T U^T$ for some set $T$ of places disjoint with $S$ where $U_T\subset \prod_{v\in T} G(F_v)$ and $U^T=\prod_{v\notin T} U_v$ with $U_v$ a subgroup of $G(F_v)$, we may define $U^{(S)}$, $U_S$,  and view them as  subgroups of $U$ similarly.  For any ideal $b$ of $\CO$, we also write $U^{(b)}$ for $U^{(S_b)}$, and $U_b$ for $U_{S_b}$,  with $S_b$ the set of places dividing $b$. Let $U_0(N)$ and $U_1(N)$ denote subgroups of $\GL_2(\wh{\CO})$ defined by
$$U_0(N)=\left\{ \matrixx{a}{b}{c}{d}\in \GL_2(\wh{\CO})\Big| c\in N\wh{\CO}\right\}, \qquad U_1(N)=\left\{ \matrixx{a}{b}{c}{d}\in U_0(N)\Big| d\equiv 1\mod N\wh{\CO}\right\}.$$
When $F$ is totally real field and $\sigma$ is an automorphic cuspidal representation of level $N$ such that $\sigma_v$ is a discrete series for all $v|\infty$, for an automorphic form $\phi$ of level $U_1(N)$, we let $(\phi, \phi)_{U_0(N)}$ denote the Petersson norm defined using the invariant measure $dx dy/y^2$ on the upper half plane.

\medskip

\subsection{The Explicit Gross-Zagier Formula}

Let $F$ be a totally real number field of degree $d$,  $\BA=\BA_F$ the \adele ring of $F$, and $\BA_f$ its finite part. Let $\BB$ be an incoherent quaternion algebra over $\BA$, totally definite at infinity. For each open compact subgroup $U$ of $\BB_f^\times=(\BB\otimes_\BA \BA_f)^\times$, let $X_U$ be the Shimura curve over $F$ associated to $U$ and $\xi_U\in \Pic(X_U)_\BQ$ the normalized Hodge class on $X_U$, i.e. the unique line bundle, which has degree one on each geometrically connected components, and is parallel to
$$\omega_{X_U/F}+\sum_{x\in X_U(\ov{F})} (1-e_x^{-1})x.$$
 Here $\omega_{X_U/F}$ is the canonical bundle of $X_U$, $e_x$ is the ramification index of $x$ in the complex uniformization of $X_U$, i.e. for a cusp $x$, $e_x=\infty$ so that $1-e_x^{-1}=1$; for a non-cusp $x$, $e_x$ is the ramification index of any preimage of $x$ in the map $X_{U'}\ra X_U$ for any sufficiently small open compact subgroup $U'$ of $U$ such that each geometrically connected component of $X_{U'}$ is a free quotient of $\CH$ under the complex uniformization. For any two open compact subgroups $U_1\subset U_2$ of $\BB_f^\times$, there is a natural surjective morphism $X_{U_1}\ra X_{U_2}$.  Let $X$ be the projective limit of the system $(X_U)_U$, which is endowed with the Hecke action of $\BB^\times$ where $\BB_\infty^\times$ acts trivially.   Note that each $X_U$ is the quotient of $X$ by the action of $U$.

Let $A$ be a simple abelian variety over $F$ parametrized by $X$ in the sense that there is a non-constant morphism $X_U\ra A$ over $F$ for some $U$. Then by Eichler-Shimura theory, $A$ is of strict $\GL(2)$-type in the sense that $M:=\End^0(A)=\End(A)\otimes_\BZ \BQ$ is a field and $\Lie(A)$ is a free  module of rank one over $M\otimes_\BQ F$ by the induced action. Let
$$\pi_A=\Hom^0_\xi (X, A):=\varinjlim_U \Hom^0_{\xi_U} (X_U, A),$$
where $\Hom^0_{\xi_U} (X_U, A)$ denotes the morphisms in $\Hom(X_U,
A)\otimes_\BZ \BQ$ using $\xi_U$ as a base point: if $\xi_U$ is
represented by a divisor $\sum_i a_i x_i$ on $X_{U, \ov{F}}$, then
$f\in \Hom_F(X_U, A)\otimes_\BZ \BQ$ is in $\pi_A$ if and only if
$\sum_i a_i f(x_i)=0$ in $A(\ov{F})_{\BQ}:=A(\ov{F})\otimes_\BZ
\BQ$.  For each open compact subgroup $U$ of $\BB_f^\times$, let $J_U$ denote the Jacobian of $X_U$. Then $\pi_A=\Hom^0(J, A):=\varinjlim_U \Hom^0(J_U, A)$ where $\Hom^0(J_U, A)=\Hom_F(J_U, A)\otimes_\BZ \BQ$.
 The action of $\BB^\times$ on $X$ induces a natural $\BB^\times$-module structure on $\pi_A$ so that $\End_{\BB^\times} (\pi_A)=M$ and has a  decomposition $\pi_A=\otimes_M \pi_{A, v}$ where $\pi_{A, v}$ are absolutely irreducible representations of $\BB_v^\times$ over $M$. Using Jacquet-Langlands correspondence, one can define the complete L-series of $\pi_A$
$$L(s, \pi_A)=\prod_v L(s, \pi_{A, v})\in M\otimes_\BQ \BC$$
as an entire function of $s\in \BC$. Let $L(s, A, M)$ denote the L-series of $\ell$-adic Galois representation with coefficients in $M\otimes_\BQ \BQ_\ell$ associated to $A$ (without local Euler factors at infinity), then $L_v(s, A, M)=L(s-\frac{1}{2}, \pi_v)$ for all finite places $v$ of $F$. Let $A^\vee$ denote the dual abelian variety of $A$. There is perfect $\BB^\times$-invariant pairing
$$\pi_A\times \pi_{A^\vee}\lra M$$
given by $$(f_1, f_2)=\Vol(X_U)^{-1} (f_{1, U}\circ f_{2, U}^\vee), \qquad f_{1, U}\in \Hom(J_U, A), \quad f_{2, U}\in \Hom(J_U, A^\vee),$$
where $f_{2, U}^\vee: A\ra J_U$ is the dual of $f_{2, U}$ composed with the canonical isomorphism $J_U^\vee \simeq J_U$. Here $\Vol(X_U)$ is defined by a fixed invariant measure on the upper half plane.  It follows that $\pi_{A^\vee}$ is dual to $\pi_A$ as representations of $\BB^\times$ over $M$. For any fixed open compact subgroup $U$ of $\BB_f^\times$, define the $U$-pairing on $\pi_A\times \pi_{A^\vee}$ by
$$(f_1, f_2)_U=\Vol(X_U) (f_1, f_2), \qquad f_1\in \pi_A, f_2\in \pi_{A^\vee}$$
which is independent of the choice of measure defining $\Vol(X_U)$. When $A$ is an elliptic curve and identify $A^\vee$ with $A$ canonically, then  for any morphism $f: X_U\ra A$, we have $(f, f)_U=\deg f$,  the degree of the finite morphism $f$.

Let $K$ be a totally imaginary quadratic extension over $F$ with
associated quadratic character $\eta$ on $\BA^\times$. Let $L$ be a finite extension of $M$ and  $\chi: K^\times\bs K_\BA^\times \ra L^\times$
an $L$-valued Hecke character  of finite order. Let $L(s, A, \chi)\in L\otimes_\BQ \BC$ be the complete L-series obtained by $\ell$-adic Galois representation associated to $A$ tensored with the induced representation of  $\chi$ from $\Gal(\ov{K}/K)$ to $\Gal(\bar{\BQ}/\BQ)$. Assume  that
$$\omega_A \cdot \chi|_{\BA^\times}=1,$$where $\omega_A$ is the central character of $\pi_A$ on $\BA_f^\times$,  and that for each finite place $v$ of $F$,
$$\epsilon(\pi_{A, v}, \chi_v)=\chi_v\eta_v(-1)\epsilon(\BB_v),$$ where $\epsilon(\BB_v)=1$ if $\BB_v$ is split and $=-1$ otherwise, and $\epsilon(\pi_
{A, v}, \chi_v)=\epsilon(1/2, \pi_{A, v}, \chi_v)$ is the local root number of $L(s, \pi_A, \chi)$.
It follows that the global root number of the L-series $L(s, \pi_A,
\chi)$ is $-1$ and there is an embedding of $K_\BA$ into $\BB$ over $\BA$. We fixed such an embedding once for all and then view $K_\BA^\times$ as a subgroup of $\BB^\times$.

Let $N$ be the conductor of $\pi^\JL$, $D$ the relative discriminant of $K$ over $F$,  $c\subset \CO$ be the ideal maximal such that $\chi$ is trivial on $\prod_{v\nmid c}\CO_{K_v}^\times \prod_{v|c} (1+c\CO_{K, v})$. Define the following sets of places $v$ of $F$ dividing $N$:
$$\Sigma_1 := \left\{ v | N \text{ nonsplit in $K$}: \ord_v(c) < \ord_v(N) \right\}.$$
Let $c_1=\prod_{\fp|c, \fp\notin \Sigma_1} \fp^{\ord_\fp c}$ be the $\Sigma_1$-off part of $c$,  $N_1$ the $\Sigma_1$-off part of $N$,  and $N_2=N/N_1$.

Let $v$ be a place of $F$ and $\varpi_v$ a uniformizer of $F_v$.  Then there exists an $\CO_v$-order $R_v$ of $\BB_v$ with discriminant $N\CO_v$ such that $R_v \cap K_v=\CO_{c_1, v}$.    Such an order $R_v$ is called admissible for $(\pi_v. \chi_v)$ if it also satisfies the following  conditions (1) and (2). Note that up to $K_v^\times$-conjugate there is a unique such order when $v\nmid (c_1, N)$, and that $\BB$ must be split at places $v|(c_1, N)$ by Lemma \ref{Tunnell-Satio}.
  \begin{enumerate}
  \item If $v|(c_1, N)$,  then $R_v$ is the intersection of two maximal orders $R_v', R_v''$ of $\BB_v$ such that
       $R_v'\cap {K_v}=\CO_{c, v}$ and  $$ R_v''\cap K_v=\begin{cases}\CO_{c/N, v}, \  &\text{ if $\ord_v(c/N)\geq 0$},\\
       \CO_{K, v}, &\text{otherwise.}\end{cases}$$
       \end{enumerate}
Note that for $v|(c_1, N)$, there is a unique order, up to $K_v^\times$-conjugate, satisfying the condition (1) unless $\ord_v(c_1)<\ord_v(N)$. In the case  $0<\ord_v(c_1)<\ord_v(N)$,  $v$ must split in $K$ by the definition of $\Sigma_1$ and there are exactly two $K_v^\times$-conjugacy classes of orders satisfying the condition (1), which are conjuagte to each other by a normalizer of $K_v^\times$ in $\BB_v^\times$. Fix an $F_v$-algebra isomorphism $K_v\cong F_v^2$ and identify $\BB_v$ with $\End_{F_v}(K_v)$. Then the two classes contain respectively orders $R_{i, v}=R'_{i, v}\cap R''_{i, v}, i=1, 2$ as in (1) such that $R'_{i, v}=\End_\CO (\CO_c), i=1, 2,$ and  $R''_{1, v}=\End_{\CO_v}((\varpi_v^{n-c}, 1)\CO_{K_v})$ and $R''_{2, v}=\End_{\CO_v}((1, \varpi_v^{n-c})\CO_{K_v})$.

 \begin{enumerate}
  \item[(2)] If $0<\ord_v(c_1)<\ord_v(N)$, then  $R_v$ is  $K_v^\times$-conjugate to some $R_{i, v}$  such that $\chi_i$ has conductor $\ord_v(c)$, where $\chi_i, i=1, 2$ is defined by $\chi_1(a)=\chi_v(a, 1)$ and $\chi_2(b)=\chi_v(1, b)$.
  \end{enumerate}
\begin{defn}An $\wh{\CO}$-order $\CR$ of $\BB_f$ is called admissible for $(\pi, \chi)$ if for every finite place $v$ of $F$, $\CR_v:=\CR\otimes_{\wh{\CO}}\CO_v$ is admissible for $(\pi_v, \chi_v)$. Note that an admissible order $\CR$ for $(\pi, \chi)$ is of discriminant $N\wh{\CO}$ such that $\CR\cap \wh{K}=\wh{\CO}_{c_1}$.
\end{defn}
Let $\CR$ be an $\wh{\CO}$-order of $\BB_f$ with discriminant $N$ such that $\CR\cap K_{\BA_f}=\wh{\CO}_{c_1}$ and that $\CR_v:=\CR\otimes_{\wh{\CO}}\CO_v$ is admissible with $\chi_v$ for any places $v$ in the above sense. Note that $\CR_v$ is unique up to $K_v^\times$-conjugate for any $v\nmid (c_1, N)$.

Let $U=\CR^\times$ and $U^{(N_2)}:=\CR^\times \cap \BB_f^{\times (N_2)}$.
Note that for any finite place $v|N_1$, $\BB_v$ must be split (by Lemma \ref{Tunnell-Satio} (5)). Let $Z\cong \BA_f^\times$ denote the center of $\BB_f^\times$. The group $U^{(N_2)}$ has a decomposition $U^{(N_2)}=U' \cdot(Z\cap U^{(N_2)})$ where $U'=\prod_{v\nmid N_2\infty} U'_v$ such that for any finite place $v\nmid N_2$,  $U'_v=U_v$ if $v\nmid N$ and $U'_v\cong {U_1(N)}_v$ otherwise.  View $\omega$ as a character on $Z$ and we may define a character on $U^{(N_2)}$  by $\omega$ on $Z\cap U^{(N_2)}$ and trivial on $U'$, which we also denoted by $\omega$.

\begin{defn} Let $V(\pi, \chi)$ denote the space of forms $f\in \pi_A\otimes_M L$, which are $\omega$-eigen under $U^{(N_2)}$, and $\chi_v^{-1}$-eigen under $K_v^\times$ for all places $v\in \Sigma_1$.  The space $V(\pi, \chi)$  is actually  a one dimensional $L$-space  (see Proposition \ref{local-multi-one}).
\end{defn}

Consider the Hecke action of $K_\BA^\times\subset \BB^\times$ on $X$. Let $X^{K^\times}$ be the $F$-subscheme of $X$ of fixed points of $X$ under $K^\times$. The theory of complex multiplication asserts that  every point in $X^{K^\times}(\bar{F})$ is defined over $K^\ab$ and that the Galois action is given by the Hecke action under the reciprocity law. Fix a point $P\in X^{K^\times}$ and let $f\in V(\pi, \chi)$ be a non-zero vector.  Define a Heegner cycle associated to $(\pi, \chi)$ to be
$$P_\chi^0(f):=\sum_{t\in \Pic_{K/F}(\CO_{c_1})} f(P)^{\sigma_t} \chi(t)\in A(K^\ab)_\BQ \otimes_M L,$$where $\Pic_{K/F}(\CO_{c_1})=\wh{K}^\times/K^\times \wh{F}^\times \wh{\CO}_{c_1}^\times$ and $t\mapsto \sigma_t$ is the reciprocity law map in the class field theory.
The \Neron-Tate height pairing over $K$ gives a $\BQ$-linear map $\pair{\ ,\ }_K: A(\bar{K})_\BQ \otimes_M A^\vee(\bar{K})_\BQ\ra \BR$. Let $\pair{\ ,\ }_{K, M}:  A(\bar{K})_\BQ \otimes_M A^\vee(\bar{K})_\BQ\ra M\otimes_\BQ\BR$ be the unique $M$-bilinear pairing such that $\pair{\ ,\ }_K=\tr_{M\otimes \BR/\BR} \pair{\ ,\ }_{K, M}$. The pairing $\pair{\ ,\ }_{K, M}$ induces an $L$-linear \Neron-Tate pairing over $K$:
$$\pair{\ ,\ }_{K, L}: (A(\bar{K})_\BQ\otimes_M L)\otimes_L (A^\vee(\bar{K})_\BQ\otimes_M L)\lra L\otimes_\BQ \BR.$$
Note that the $\BB^\times$-invariant $M$-linear pairing $(\ , \ )_U: \pi_A\times \pi_{A^\vee}\ra M$ induces a $\BB^\times$-invariant $L$-linear pairing
$$(\ ,\ )_U: (\pi_A\otimes_M L)\times (\pi_{A^\vee}\otimes_M L)\lra L.$$

The Hilbert newform $\phi$ in the Jacquet-Langlands correspondence $\sigma$ of $\pi_A$ on $\GL_2(\BA)$ is the form of level $U_1(N)$, for each $v|\infty$, $\SO_2(\BR)\subset \GL_2(F_v)$ acts by the character $\sigma(k_\theta)\phi=e^{4\pi i \theta}\phi$ where  $k_\theta=\matrixx{\cos\theta}{\sin\theta}{-\sin\theta}{\cos\theta}\in \SO_2(\BR)$, such that
$$L(s, \pi)=2^d\cdot |\delta|_\BA^{s-\frac{1}{2}}\cdot Z(s, \phi), \quad Z(s, \phi)=\int_{F^\times\bs \BA^\times} \phi\matrixx{a}{}{}{1} |a|_\BA^{s-\frac{1}{2}}d^\times a$$
where the measure $d^\times a$ is taken to be Tamagawa measure so that $\Res_{s=1} \int_{|a|\leq 1, a\in F^\times\bs \BA^\times} |a|^{s-1}d^\times a=\Res_{s=1}L(s, 1_F)$, and $\delta$ is the differential of $F$. Note that $\phi(g)\ov{\phi}(g)$ is a function on $${GL_2(F)}_+\bs {\GL_2(F_\infty)}_+\times \GL_2(\BA_f)/Z(\BA)\cdot (U_{1, \infty}\times U_0(N)) \cong {\GL_2(F)}_+\bs \CH^d \times \GL_2(\BA_f)/U_0(N) \BA_f^\times.$$ We define the Petersson norm $(\phi, \phi)_{U_0(N)}$ by the integration of $\phi \ov{\phi}$ with measure $dx dy/y^2$ on each upper half plane. One main result of this paper is the following.

\begin{thm}[Explicit Gross-Zagier Formula] \label{GZ} Let $F$ be a totally real field of degree $d$. Let $A$ be an abelian variety over $F$ parametrized by a Shimura curve $X$ over $F$ and $\phi$ the Hilbert holomorphic newform of parallel weight $2$ on $\GL_2(\BA)$ associated to $A$. Let $K$ be a totally imaginary quadratic extension over $F$ with relative discriminant $D$ and discriminant $D_K$. Let $\chi: K_\BA^\times/K^\times \ra L^\times$  be a finite Hecke character of conductor $c$ over some finite extension $L$ of $M:=\End^0(A)$. Assume that
 \begin{enumerate}
 \item $\omega_A \cdot \chi|_{\BA^\times}=1$, where $\omega_A$ is the central character of $\pi_A$;
 \item for any place $v$ of $F$, $\epsilon(\pi_{A, v}, \chi_v)=\chi_v \eta_v(-1)\epsilon(\BB_v)$.
 \end{enumerate}
For any non-zero forms $f_1\in V(\pi_A, \chi)$ and $f_2\in V(\pi_{A^\vee}, \chi^{-1})$, we have an equality in $L\otimes_\BQ \BC$:
$$L^{'(\Sigma)} (1, A, \chi)=2^{-\#\Sigma_D} \cdot \frac{(8\pi^2)^d\cdot(\phi, \phi)_{U_0(N)}}{u_1^2 \sqrt{|D_K| \|c_1^2\|}}\cdot \frac{\pair{P_\chi^0(f_1), P^0_{\chi^{-1}}(f_2)}_{K, L}}{(f_1, f_2)_{\CR^\times}}$$
where
 $$\begin{aligned}
 &\Sigma:=\left\{v|(N, Dc): \text{if $v\|N$ then $\ord_v(c/N)\geq 0$}\right\},\\
&\Sigma_D := \left\{ v| (N,D): \ord_v(c) < \ord_v(N)  \right\},
\end{aligned}$$
the ideal $c_1|c$ is the $\Sigma_1$-off part of $c$ as before,   $u_1=\#\kappa_{c_1}\cdot [\CO_{c_1}^\times: \CO^\times]$ and $\kappa_{c_1}$ is the kernel of the morphism from $\Pic(\CO)$ to $\Pic(\CO_{c_1})$ which has order $1$ or $2$, and $(\phi, \phi)_{U_0(N)}$ is the Petersson norm  with respect to the measure $dxdy/y^2$ on the upper half plane.
\end{thm}

\s{\bf Remark}: Note that the assumption  $\omega_A|_{\BA^\times} \cdot \chi=1$ implies  $L(s, A, \chi)=L(s, A^\vee, \chi^{-1})$. Let $\phi^\vee$ be the Hilbert newform associated to $A^\vee$, then $(\phi^\vee, \phi^\vee)_{U_0(N)}=(\phi, \phi)_{U_0(N)}$.

We may state the above theorem in simpler way under some assumptions.  Assume that
\begin{itemize}
\item $\omega_A$ is unramified, and if $v\in \Sigma_1$ then $v\nmid c$;
\end{itemize}
Note that $c_1=c$ under the above assumption. Fix an infinite place $\tau$ of $F$ and let $B$ be nearby quaternion algebra whose ramification set is obtained from the one of $\BB$ by removing $\tau$. Then there is an $F$-embedding of $K$ into $B$ which we fix once for all and view $K^\times$ as a $F$-subtorus of $B^\times$. Let $R$ be an admissible  $\CO$-order of $B$ for $(\pi, \chi)$, by which we mean that $\wh{R}$ is an admissible $\wh{\CO}$-order of $\BB_f=\wh{B}$ for $(\pi, \chi)$. Note that $R$ is of discriminant $N$ such that $R\cap K=\CO_c$.   Let $U=\wh{R}^\times \subset \wh{B}^\times$ and let $X_U$ be the Shimura curve of level $U$ so that it has complex uniformization
$$X_{U, \tau} (\BC)=B^\times_+\bs \CH \times \wh{B}^\times/ U \cup \{\mathrm{Cusps}\},$$where $B_+^\times$ is the subgroup of elements $x\in B^\times$ with totally positive norms. Let $u=\#\kappa_c\cdot [\CO_c^\times: \CO^\times]$. By Proposition \ref{local-multi-one-2}, we have that $V(\pi_A, \chi)\subset (\pi_A\otimes_M L)^{\wh{R}^\times}$.

\s{\em Special case 1}. {\em Further assume that $(N, Dc)=1$}.  Then there is a non-constant morphism $f: X_U \lra A$ mapping a Hodge class on $X_U$ to torsion of $A$ and for any two such morphisms $f_1, f_2: X_U\ra A$, $n_1f_1=n_2f_2$ for some non-zero integers $n_1, n_2$. Let $P=[h_0, 1]\in X_U$ be the point with $h_0$ the unique fixed point of $K^\times$.  Replace $\chi$ by $\chi^{-1}$, there is a non-constant morphism $X_U\ra A^\vee$ with similar uniqueness.  For any such $f_1: X_U\ra A$ and $f_2: X_U\ra A^\vee$, let $(f_1, f_2)=f_1\circ f_2^\vee$. Then we have an equality in $L\otimes_\BQ \BC$:
$$L' (1, A, \chi)=\frac{(8\pi^2)^d(\phi, \phi)_{U_0(N)}}{u^2 \cdot\sqrt{|D_K| \|c^2\|}}\cdot \frac{\pair{P_\chi^0(f_1), P^0_{\chi^{-1}}(f_2)}_{K, L}}{(f_1, f_2)_U}.$$

\s{\em Special case 2}. {\em Further assume that $\omega_A$ is trivial, or more general, that $\omega_A(\varpi_v)\in \Aut(A)^2\subset M^{\times 2}$ for all places $v|(N, D)$ but $v\nmid c$, where $\varpi_v$ is a uniformizer of $F_v$.}   For each place $v|(N, D)$ but $v\nmid c$, $K_v^\times$ normalizes $R_v^\times$ (See Lemma
\ref{order-uniqueness}) and a uniformizer $\varpi_{K_v}$ of $K_v$ induces an automorphism $T_{\varpi_{K_v}}: X_U\ra X_U$ over $F$. Note that $\chi_v(\varpi_{K_v})\in \Aut(A)\subset M^\times$. There exists a non-constant morphism $f: X_U\ra A$ mapping a Hodge class to torsion point such that $T_{\varpi_{K_v}}f =\chi^{-1}(\varpi_{K_v}) f$
for each place $v|(N, D)$ but $v\nmid c$. Such $f$ has the same uniqueness property as in special case 1. Then for any such $f_1: X_U\ra A$ and $f_2: X_U\ra A^\vee$, we have an equality in $L\otimes_\BQ \BC$:
$$L^{'(\Sigma)} (1, A, \chi)=2^{-\#\Sigma_D}\cdot\frac{(8\pi^2)^d(\phi, \phi)_{U_0(N)}}{u^2 \cdot\sqrt{|D_K| \|c^2\|}}\cdot \frac{\pair{P_\chi^0(f_1), P^0_{\chi^{-1}}(f_2)}_{K, L}}{(f_1, f_2)_U},$$
where $\Sigma $ is now the set of places $v|(cD,N)$ of $F$ such that if $v\| N$ then $v\nmid D$.

\s{\bf Example}. Let $\phi\in S_2(\Gamma_0(N))$ be a newform. Let $K$ be an imaginary quadratic field of discriminant $D$ and $\chi$ a primitive character of $\Pic(\CO_c)$. Assume that $(\phi, \chi)$ satisfies the Heegner condition (1)-(2) in Theorem \ref{A}, then by Lemma \ref{Tunnell-Satio} (1) and (3), $\epsilon(\phi, \chi)=-1$ and  $B=M_2(\BQ)$. The condition (1)-(2) also implies that there exists $a, b\in \BZ$ with $(N, a, b)=1$ such that $a^2-4Nb=Dc^2$. Fix an embedding of  $K$ into $B$ by
$$(Dc^2+\sqrt{Dc^2})/2\lto \matrixx{(Dc^2+a)/2}{-1}{Nb}{(Dc^2-a)/2}.$$
Then $R:=\left\{\matrixx{a}{b}{c}{d}\in M_2(\BZ) \Big| N|c\right\}$ is an order of $B$ such that $\wh{R}\cap K=\CO_c$. Let $A$ be an abelian variety associated to $\phi$ via Eichler-Shimura theory and $f: X_0(N)\ra A$ be any non-constant morphism mapping cusp $\infty$ to $O\in A$. Then $f\in V(\pi_A, \chi)$. Let $z\in \CH$ be the fixed point by $K^\times$, then $Nbz^2-az+1=0$, $\CO_c=\BZ+\BZ z^{-1}$, and $\fn^{-1}=\BZ +\BZ N^{-1} z^{-1}$ so that $\CO_c/\fn \cong \BZ/N\BZ$. The point on $X_0(N)$ corresponding $z$ via complex uniformation represents the isogeny $\BC/(\BZ +\BZ z)\ra \BC/(N^{-1} \BZ +\BZ z)$, or $\BC/\CO_c\ra \BC/\fn^{-1}$. Thus Theorem \ref{A} now follows from Theorem \ref{GZ}.

\medskip

For various arithmetic applications, we may need explicit formula for different test vectors. We now give variations of the explicit formula for different test vectors. Let $v$ be a finite place of $F$, fix $\pair{\ ,\ }_v$ a $\BB_v^\times$-invariant pairing on $\pi_{A, v}\times \pi_{A^\vee, v}$ and a Haar measure $dt_v$ on $F_v^\times\bs K_v^\times$. For any $f'_{1, v}\in \pi_{A, v}, f'_{2, v}\in \pi_{A^\vee, v}$ with $\pair{f'_{1, v}, f'_{2, v}}_v\neq 0$, let
$$\beta^0(f'_{1, v}, f'_{2, v})=\beta^0(f'_{1, v}, f'_{2, v}, dt_v)=\int_{F_v^\times\bs K_v^\times} \frac{\pair{\pi_{A, v}(t_v)f'_{1, v}, f'_{2, v}}_v}{\pair{f'_{1, v}, f'_{2, v}}_v} \chi_v(t_v) dt_v.$$
For any two non-zero pure tensor forms $f'=\otimes_v f_v', f''=\otimes_v f''_v\in \pi$, we say that $f', f''$  differ (resp. coincide) at a place $v$ if $f_v'$ and $f_v''$ are not parallel (resp. are parallel). It is independent of the decompositions. In particular, if two non-zero pure tensor forms coincide locally everywhere then they are the same up to a scalar.
\begin{thm}[Variation of Gross-Zagier Formula]\label{variation1}
  Let $(A, \chi)$  and $f_1\in V(\pi_A,  \chi), f_2\in V(\pi_{A^\vee}, \chi^{-1})$ be as in Theorem \ref{GZ}. Let $S$ be a finite set of finite places of $F$,  $f_1'\in\pi_A, f_2'\in \pi_{A^\vee}$ be vectors such that $f_i'$ and $f_i$  coincide for any $v\notin S, i=1, 2$,  and $\pair{f'_{1, v}, f'_{2, v}}_v\neq 0$ and $\beta^0(f'_{1, v}, f'_{2, v})\neq 0,$ for any $v\in S$.  Define
 $$P^0_\chi (f_1')=\frac{\#\Pic(\CO_{c_1})}{\Vol(K^\times\wh{F}^\times\bs \wh{K}^\times, dt)}\cdot \int_{K^\times\wh{F}^\times\bs \wh{K}^\times} f'_1(P)^{\sigma_t} \chi(t) dt,$$
 and define $P^0_{\chi^{-1}}(f_2')$ similarly.  Then,  with notations as in Theorem \ref{GZ},  we have that
$$ L^{'(\Sigma)} (1, A, \chi)=2^{-\#\Sigma_D} \cdot \frac{(8\pi^2)^d\cdot (\phi, \phi)_{U_0(N)}}{u_1^2 \sqrt{|D_K| \|c_1^2\|}}\cdot \frac{\pair{P_\chi^0(f'_1), P^0_{\chi^{-1}}(f'_2)}_{K, L}}{(f'_1, f'_2)_\CR^\times}\cdot \prod_{v\in S}\frac{\beta^0(f_{1, v}, f_{2, v})}{ \beta^0(f'_{1, v}, f'_{2, v})},$$which is independent of the choice of Haar measure $dt_v$ for $v\in S$.
\end{thm}

\s{\bf Example}. Let $A$ be the elliptic curve $X_0(36)$ with the cusp $\infty$ as the identity point and let $K=\BQ(\sqrt{-3})$. Let $p\equiv 2\mod 9$ be a prime, then the field $L'=K(\sqrt[3]{p})$ is contained in $H_{3p}$. Let $\chi: \Gal(L'/K)\ra K^\times$ be the character mapping $\sigma$ to $(\sqrt[3]{p})^{\sigma-1}$. Fix the embedding $K\ra M_2(\BQ)$ mapping $w:=(-1+\sqrt{-3})/2$ to $\matrixx{-1}{-p/6}{6/p}{0}$.

For $f'=\id: X_0(36)\ra A$, let $P\in X_0(36)$ be the point corresponding to $-pw/6\in \CH$.  The Heegner divisor $P^0_\chi(f')$ is
 $$P^0_\chi(f')=\frac{1}{9}\sum_{t\in \Pic(\CO_{6p})} f'(P)^{\sigma_t}\chi(t).$$
One can show that $P^0_\chi(f')$ is non-trivial (see \cite{Satge}, \cite{DV} and \cite{CST}) and then it follows that the prime $p$ is the sum of two rational cubes. By the variation formula, one can easily obtain the height formula of $P^0_\chi(f')$: let $\phi\in S_2(\Gamma_0(36)$ be the newform associated to $A$, and note that $\#\Sigma_D=1$, $u_1=1$, and $c_1=p$ in the variation,
$$L^{'(\infty)}(1, A, \chi)=9\cdot \frac{8\pi^2\cdot (\phi, \phi)_{\Gamma_0(36)}}{\sqrt{3 p^2}} \cdot \pair{P^0_\chi(f'), P^0_{\chi^{-1}}(f')}_{K, K}.$$

In fact,  $U=\CR^\times$ in Theorem \ref{GZ} is given by
$$\CR=\left\{ \matrixx{a}{b/6}{6c}{d}\in M_2(\wh{\BQ})\Big| a, b, c, d\in \wh{\BZ}, p^{-1}b+pc, a+pc-d\in 6\wh{\BZ}\right\}$$
and $f\in V(\pi_A, \chi)$ is $\chi_v^{-1}$-eigen  for $v=2, 3$. Then $(f', f')=\Vol(X_U)/\Vol(X_0(36))=2/9$. The ratio $\displaystyle{\frac{\beta^0(f_v, f_v)}{\beta^0(f'_v, f'_v)}}$ is equal to $1$ at $v=2$ and $4$ at $v=3$.

\medskip

\subsection{The Explicit Waldspurger Formula}
Let  $F$ be a general base number field. Let  $B$ be a quaternion algebra over $F$ and $\pi$ a cuspidal  automorphic representation of $B_\BA^\times$ with central character $\omega$.   Let $K$ be a quadratic field extension of $F$ and $\eta$ the quadratic Hecke character on $F^\times\bs \BA^\times$ associated to the quadratic extension. Let $\chi$ be a Hecke chareacter on $K_\BA^\times$. Write $L(s, \pi, \chi)$ for the Rankin L-series $L(s, \pi^\JL \times \pi_\chi)$, where $\pi^\JL$ is the Jacquet-Langlands correspondence of $\pi$ on $\GL_2(\BA)$ and $\pi_\chi$ the automorphic representation of $\GL_2(\BA)$ corresponding to theta series of $\chi$ so that $L(s, \pi_\chi)=L(s, \chi)$.  Assume that
$$\omega \cdot \chi|_{\BA^\times}=1.$$ Then for any place $v$ of $F$, the local root number $\epsilon(1/2, \pi_v, \chi_v)$ of the Rankin L-series is independent of the choice of additive character. We also assume that for all places $v$ of $F$
$$\epsilon(1/2, \pi_v, \chi_v)=\chi_v\eta_v(-1)\epsilon(B_v), $$
 where $\epsilon(B_v)=-1$ if $B_v$ is division and $+1$ otherwise. It follows that the global root number $\epsilon(1/2, \pi, \chi)=+1$ and there exists an $F$-embedding of $K$ into $B$. We fix such an embedding once for all and view $K^\times$ as an $F$-subtorus of $B^\times$.

Let $N$ be the conductor of $\pi^\JL$, $D$ the relative discriminant of $K$ over $F$,  $c\subset \CO$ be the ideal maximal such that $\chi$ is trivial on $\prod_{v\nmid c}\CO_{K_v}^\times \prod_{v|c} (1+c\CO_{K, v})$. Define the following sets of places $v$ of $F$ dividing $N$:
$$\Sigma_1 := \left\{ v |N \text{ nonsplit in $K$}: \ord_v(c) < \ord_v(N) \right\},$$
Let $c_1=\prod_{\fp|c, \fp\notin \Sigma_1} \fp^{\ord_\fp c}$ be the $\Sigma_1$-off part of $c$, $N_1$ the $\Sigma_1$-off part of $N$, and  $N_2=N/N_1$ be the $\Sigma_1$-part of $N$.

Let $R$ be an admissible $\CO$-order of $B$ for $(\pi, \chi)$ in the sense that $R_v$ is admissible for $(\pi_v, \chi_v)$ for every finite place $v$ of $F$. It follows that $R$ is $\CO$-order with discriminant $N$ such that $R\cap K=\CO_{c_1}$.

Let $U=\prod_v U_v\subset B_\BA^\times$ be a compact subgroup satisfying that for any finite place $v$, $U_v=R_v^\times$,  and that for any infinite place $v$ of $F$,  $U_v$ is a maximal compact subgroup of $B_v^\times$ such that $U_v\cap K_v^\times$ is the maximal compact subgroup of $K_v^\times$.
Note that for any finite place $v|N_1$, $B_v$ must be split. Let $Z\cong \BA_f^\times$ denote the center of $\wh{B}^\times$. The group $U^{(N_2 \infty)}$ has a decomposition $U^{(N_2\infty)}=U' \cdot(Z\cap U^{(N_2 \infty)})$ where $U'=\prod_{v\nmid N_2\infty} U'_v$ such that for any finite place $v\nmid N_2$,  $U'_v=U_v$ if $v\nmid N$ and $U'_v\cong {U_1(N)}_v$ otherwise.  View $\omega$ as a character on $Z$ and we may define a character on $U^{(c_2\infty)}$  by $\omega$ on $Z\cap U^{(c_2\infty)}$ and trivial on $U'$, which we also denote by $\omega$.

\begin{defn}\label{test space} Let $V(\pi, \chi)$ denote the space of forms $f=\otimes_v f_v\in \pi$ such that $f$ is $\omega$-eigen under $U^{(N_2 \infty)}$;  for all places $v\in \Sigma_1$,  $f$ is $\chi_v^{-1}$-eigen under $K_v^\times$; and for any infinite place $v$, $f$ is $\chi_v^{-1}$-eigen  under $U_v\cap K_v^\times$ with weight minimal. The space $V(\pi, \chi)$ is actually  a one dimensional space (see Proposition \ref{local-multi-one}).
\end{defn}

Let $r, s, t$ be integers such that $B\otimes_\BQ \BR=\BH^r\times M_2(\BR)^s\times M_2(\BC)^t$, and let $X_U$ denote  the $U$-level real manifold $$X_U=B_+^\times\bs (\CH_2^s\times \CH_3^t)\times
\wh{B}^\times/ U,$$  which has finitely many connected components, where $\CH_2, \CH_3$ are the usual hyperbolic spaces of dimension two and three respectively. Define
the volume of $X_U$, denoted by $\Vol(X_U)$, as follows.
\begin{itemize}
\item if $s+t>0$, then $X_U$ is disjoint union of dimension $2s+3t$
manifolds:
$$X_U=B_+^\times\bs (\CH_2^s\times \CH_3^t)\times \wh{B}^\times
/U=\bigsqcup_i \Gamma_i\bs (\CH_2^s\times \CH_3^t),$$ for some discrete
subgroup $\Gamma_i\subset B_+^\times \cap\prod_{v|\infty, \text{$B_v$ not division}} (B_v)^\times$, then define volume of $X_U$ with the
measure $dxdy/(4\pi y^2)$ on
$\CH_2$ and the measure $dxdydv/\pi^2v^3$ on $\CH_3$. Here the notation $\CH_3$  is the same as in  \cite{Vig}.
\item if $s+t=0$, then $F$ is totally real and $B$ is totally definite. For any open compact subgroup $U$ of $\wh{B}^\times$, the double coset $B^\times \bs \wh{B}^\times/U$ is finite, let $g_1, \cdots, g_n\in \wh{B}^\times$ be a complete set of representatives for the coset. Let $\mu_Z=\wh{F}^\times \cap U$, then for any $g\in \wh{B}^\times$, $B^\times \cap g U g^{-1}/\mu_Z$ is a finite set. Define the volume of $X_U$ to be the Mass of $U$
    $$\Vol(X_U)=\Mass(U)=\sum_{i=1}^n \frac{1}{\#(B^\times\cap g_i U g_i^{-1})/\mu_Z}.$$
\end{itemize}
For any automorphic forms $f_1 \in \pi$ and $f_2 \in \tilde{\pi}$,
$\pair{f_1, f_2}_\Pet$ is
the Petersson pairing of $f_1, f_2$ defined by $$\displaystyle{\pair{f_1,
f_2}_\Pet=\int_{B^\times \BA^\times\bs B_\BA^\times} f_1(g)
f_2(g) dg,}$$ where $dg$ is the Tamagawa measure on
$F^\times\bs B^\times$ so that $B^\times\BA^\times\bs B_\BA^\times$
has total volume $2$. For any $f_1\in V(\pi,\chi)$ and
$f_2\in V(\wt{\pi},\chi^{-1})$,
one may define the $U$-level pairing as
$$\pair{f_1, f_2}_U=\frac{\pair{f_1, f_2}_\Pet}{2}\cdot \Vol(X_U).$$
For any $f\in
V(\pi, \chi)$,
define the $c_1$-level period of $f\in V(\pi, \chi)$ as follows. Let $\ov{K_\infty^\times/F_\infty^\times}$ be the closure of $K_\infty^\times/F_\infty^\times$ in the compact group
$K_\BA^\times/\BA^\times K^\times$  and endowed on $\ov{K_\infty^\times/F_\infty^\times}$ the Haar measure $dh$ of total volume one, let
$$P^0_\chi(f)=\sum_{t\in \Pic_{K/F}(\CO_{c_1})} f^0(t)\chi(t), \qquad f^0(t)=\int_{\ov{K_\infty^\times/F_\infty^\times}} f(th) \chi(h) dh.$$
Note that the function $f^0(t)\chi(t)$ on
$K_\BA^\times$ is constant on $K_{\Sigma_1}^\times$ and then can be
viewed as a function on $\Pic_{K/F}(\CO_{c_1})= \wh{K}^\times/K^\times\wh{F}^\times\wh{\CO}_{c_1}^\times$. Note that when $F$ is totally real and all infinite places $v$ of $F$ are inert in $K$, $f^0=f$.

\s{\bf Notations}: Let $b$ be an integral ideal of $F$, we define the relative regulator $R_b$ to be the quotient of the regulator of $\CO_b^\times$ by the regulator of $\CO^\times$ and $w_b=\#\CO^\times_{b, \tor}/\# \CO^\times_\tor$. Denote by $\kappa_b$ the kernel of the natural homomorphism from $\Pic(\CO)$ to $\Pic(\CO_b)$.  Define $\nu_b=2^{-r_{K/F}} R_b^{-1}\cdot\#\kappa_b\cdot w_b$ where $r_{K/F}=\rank \CO_K^\times-\rank \CO^\times$. For example, if $F$ is a totally real field of degree $d$ and $K$ is a totally imaginary quadratic field extension over $F$, then $\nu_b= 2^{1-d}\cdot \#\kappa_b\cdot [\CO_b^\times: \CO^\times]$, where $\kappa_b\subset \kappa_1$ and $\#\kappa_1=1$ or $2$ by Theorem 10.3 of \cite{Washington}.

\medskip

For an infinite place $v$ of $F$, let $U_v$ denote the maximal compact subgroup of $\GL_2(F_v)$, which is $O_2$ if $v$ is real and $U_2$ if $v$ is complex, and  let $U_{1, v}\subset U_v$ denote its subgroup of diagonal matrices $\matrixx{a}{}{}{1}$ for $a\in F_v^\times$ with $|a|_v=1$. For a generic $(\fg_v, U_v)$-module $\sigma_v$ and a non-trivial additive character $\psi_v$ of $F_v$, let $\CW(\sigma_v, \psi_v)$ be the $\psi_v$-Whittaker model of $\sigma_v$. There is an invariant bilinear pairing on $\CW(\sigma_v, \psi_v)\times \CW(\wt{\sigma}_v, \psi^{-1})$,
$$\langle W_1, W_2 \rangle_v := \int_{F_v^\times} W_1\left[
  \begin{pmatrix}
    a & \\
      & 1
  \end{pmatrix}
\right]
W_2\left[
  \begin{pmatrix}
    a & \\
      & 1
  \end{pmatrix}
\right] d^\times a$$
with the measure $d^\times a = L(1,1_v)\frac{da}{|a|_v}$ where $da$ equals $[F_v:\BR]$ times the
usual Lebesgue measure on $F_v$.  Let  $W_0\in \CW(\sigma_v, \psi_v)$ be the vector invariant under $U_{1, v}$ with minimal weight such that
$$L(s,\pi_v) = Z(s,W_0), \quad Z(s,W_0) := \int_{F_v^\times}
  W_{\sigma_v}\left[
    \begin{pmatrix}
      a & \\
        & 1
    \end{pmatrix}
  \right]|a|_v^{s-1/2} d^\times a$$
where $d^\times a$ is Tamagawa measure. Similarly define $\wt{W}_0$ for $\wt{\sigma}_v$.  Then $\Omega_{\sigma_v} := \langle W_0, \wt{W}_0 \rangle_v$  is an invariant of $\sigma_v$ which is independent of the choice of $\psi_v$ (see an explicit formula for $\Omega_{\sigma_v}$  before Lemma \ref{archimedean} ). We associate $(\sigma_v, \chi_v)$  a  constant by
  \begin{equation} \label{C}
   C(\sigma_v, \chi_v) :=
    \begin{cases}
      2^{-1}\pi\cdot \Omega_{\sigma_v}^{-1}, \quad &\text{if $K_v$ nonsplit}; \\
      \Omega_{\sigma_v\otimes \chi_{1,v}}\cdot \Omega_{\sigma_v}^{-1}, \quad &\text{if $K_v$ split},
  \end{cases}
    \end{equation}
  where for split $K_v\cong F_v^2$,  embedded into $M_2(F_v)$ diagonally, the character $\chi_1$ is given by $\chi_{1,v}(a) := \chi_v\left[
    \begin{pmatrix}
      a & \\
        & 1
    \end{pmatrix}
  \right]$. If $v$ is a  real place  of $F$ and $\sigma_v$ is a discrete series of weight $k$, then  $C(\sigma_v, \chi_v)=
 4^{k-1} \pi^{k+1}\Gamma(k)^{-1}$ when $K_v\cong \BC$, and
$C(\sigma_v, \chi_v)= 1$ when $K_v\cong \BR^2$.

Let $\sigma$ be the Jacquet-Langlands correspondence of $\pi$ to $\GL_2(\BA)$, the normalized new vector $\phi^0=\otimes_v \phi_v \in \sigma$ is the one fixed by $U_1(N)$ and  $\phi_v$ is fixed by $U_{1, v}$ with weight minimal for all $v|\infty$ such that
$$L(s, \sigma)=|\delta|_\BA^{s-\frac{1}{2}}Z(s, \phi^0), \qquad Z(s, \phi^0):=\int_{F^\times \bs \BA^\times} \phi^0\matrixx{a}{}{}{1} |a|_\BA^{s-\frac{1}{2}} d^\times a,$$
with Tamagawa measure on $\BA^\times$ so that  $\Res_{s=1}\int_{|a|\leq 1, a\in F^\times \bs \BA^\times} |a|^{s-1}d^\times a=\Res_{s=1} L(s, 1_F)$. Note that when $F$ is a totally real field and $\sigma$ a cuspidal automorphic representation such that $\sigma_v$ is discrete series for any infinite place $v$, the normalized new vector $\phi^0$ is not parallel to the Hilbert newform $\phi$: they are different at infinity.  Note that if $\sigma$ is unitary and $\phi^0$ is the normalized new vector of $\sigma$, then $\bar{\sigma}\cong \wt{\sigma}$ and $\ov{\phi^0}$ is the normalized new vector of $\bar{\sigma}$.  We will see that $(\phi, \phi)_{U_0(N)}=(2\pi)^d\pair{\phi_0, \ov{\phi}_0}_{U_0(N)}$.

\begin{thm}[Explicit Waldspurger Formula] \label{W} Let $F$ be a number field. Let $B$ be a quaternion algebra over $F$ and $\pi$ an irreducible cuspidal automorphic representation of $B_\BA^\times$ with central
character $\omega$. Let $K$ be a quadratic field extension of $F$ and $\chi$ a Hecke character of $K_\BA^\times$. Assume that
\begin{enumerate}
\item $\omega \cdot \chi|_{\BA^\times}=1$,
\item  $\epsilon(1/2, \pi_v, \chi_v)=\chi_v\eta_v(-1)\epsilon(B_v)$ for all places $v$ of $F$.
\end{enumerate}
Then for any non zero forms  $f_1 \in V(\pi,\chi)$ and $f_2 \in V(\wt{\pi},\chi^{-1})$, we have
  \[L^{(\Sigma)}(1/2,\pi,\chi) =
    2^{-\#\Sigma_D+2}\cdot C_\infty \cdot \frac{\pair{\phi^0_1,\phi^0_2}_{U_0(N)}}
    {\nu_{c_1}^2\sqrt{|D_K| \|c_1\|^{2}}}
    \cdot
  \frac{P^0_\chi(f_1)P^0_{\chi^{-1}}(f_2)}{\pair{f_1,f_2}_{\wh{R}^\times}},\]
 where $\phi_1^0 \in \pi^\JL$ and $\phi_2^0\in \wt{\pi}^\JL$ are normalized new vectors, $\Sigma $ is the set of places $v|(cD,N)\infty$ of $F$ such that if $v\| N$ then $\ord_v(c/N)\geq 0$ and  if $v|\infty$ then $K_v\cong \BC$. The constant $C_\infty=\prod_{v|\infty} C_v$, $c_1|c$ and $\Sigma_D$ are the same as in Theorem \ref{GZ},  and $C_v=C(\pi_v^\JL, \chi_v)$ is given in \eqref{C}.
\end{thm}
For many applications, we need explicit form of Waldspurger formula for different test vectors. The following variation formula is useful. For each place $v$ of $F$, fix a $B_v^\times$-invariant pairing $\pair{\ ,\ }_v$ on $\pi_v \times \wt{\pi}_v$.  Here if $v|\infty$, we mean it is the restriction of a $B_v^\times$-invariant pairing on the corresponding smooth representations. For any $f'_{1, v}\in \pi_v, f'_{2, v} \in \wt{\pi}_v$ with $\pair{f'_{1, v}, f'_{2, v}}\neq 0$, define $\beta^0(f'_{1, v}, f'_{2, v})$ same as in Theorem \ref{variation1}.

\begin{thm}[Variation of Waldspurger Formula]\label{variation2}
  Let $(\pi, \chi)$  and $f_1\in V(\pi,  \chi), f_2\in V(\wt{\pi}, \chi^{-1})$ be as in Theorem \ref{W}. Let $S$ be a finite set of places of $F$,  $f_1'\in\pi, f_2'\in \wt{\pi}$ be pure vectors which coincide  with $f_1, f_2$ outside $S$ respectively,  such that $\pair{f'_{1, v}, f'_{2, v}}_v\neq 0$ and $\beta^0(f'_{1, v}, f'_{2, v})\neq 0$ for all $v\in S$. Here $\beta^0$ is similarly defined as in Theorem \ref{variation1}.   Define
 $$P^0_\chi (f_1')=\frac{\#\Pic_{K/F}(\CO_{c_1})}{\Vol(K^\times\BA^\times\bs K_\BA^\times, dt)}\cdot \int_{K^\times\BA^\times\bs K_\BA^\times} f'(t) \chi(t) dt,$$
 and define $P^0_{\chi^{-1}}(f_2')$ similarly.  Then we have that
$$ L^{(\Sigma)} (1/2, \pi, \chi)=2^{-\#\Sigma_D+2}\cdot C_\infty \cdot \frac{\pair{\phi_1^0, \phi^0_2}_{U_0(N)}}{\nu_{c_1}^2 \sqrt{|D_K| \|c_1^2\|}}\cdot \frac{P_\chi^0(f'_1) P^0_{\chi^{-1}}(f'_2)}{\pair{f'_1, f'_2}_{\wh{R}^\times}}\cdot \prod_{v\in S} \frac{\beta^0(f_{1, v}, f_{2, v})}{ \beta^0(f'_{1, v}, f'_{2, v})},$$
where the notations are the same as in Theorem \ref{W}.
\end{thm}

\s{\bf Example}. Let $\phi=\sum a_n q^n\in S_2(\Gamma_0(N))$ be a newform of weight $2$ and $p$ a good ordinary prime of $\phi$, $K$ an imaginary quadratic field of discriminant $D$ and $\chi$ a character of $\Gal(H_c/K)$ of conductor $c$ prime to $p$. Assume that the conditions (i)-(ii) in Theorem 1.2 are satisfied.  Let $B$ be the quaternion algebra, $\pi$ the cuspidal automorphic representation on $B_\BA^\times$ and identify $\wt{\pi}$ with $\ov{\pi}$, and $f\in \pi^{\wh{R}^\times}=V(\pi, \chi)$ a non-zero test vector as in Theorem \ref{W}. Define the $p$-stabilization of $f$ by
$$f^\dag=f-\alpha^{-1} \pi\matrixx{1}{}{}{p} f$$
where $\alpha$ is the unit root of $X^2-a_pX+p$ and let $\beta=p/\alpha$ is another root. By the above variation formula and Theorem \ref{B}, one may easily obtain formula for $P_\chi^0(f^\dag)$  which is used to give interpretation property of anticyclotomic $p$-adic L-function.
$$L (1, \phi, \chi)=2^{-\mu(N, D)}\cdot \frac{8\pi^2 (\phi, \phi)_{\Gamma_0(N)}}{[\CO_c^\times: \BZ^\times]^2 \sqrt{|Dc^2|}}\cdot \frac{|P_\chi^0(f^\dag)|^2}{(f^\dag, f^\dag)_{\wh{R}^\times}}\cdot e_p,$$
where
$$e_p=\frac{\beta^0(W, \ov{W})}{\beta^0(W^\dag, \ov{W^\dag})}=\frac{L(2, 1_p)}{L(1, \pi_p, \ad)}\cdot (1-\alpha ^{-1}\chi_1(p))^{-1}(1-\beta^{-1}\chi_1^{-1}(p))^{-1}.$$
 Here $W$ is a new vector of the Whittaker model $\CW(\pi_p, \psi_p)$ with $\psi_p(x)=e^{-2\pi i\iota(x)}$ where $\iota: \BQ_p/\BZ_p\ra \BQ/\BZ$ is the natural embedding, and  $W^\dag:= W-\alpha^{-1}\pi_p\matrixx{1}{}{}{p} W$ is its stabilization, where $K_p^\times\cong \BQ_p^{\times 2}$ is embedded into $\GL_2(\BQ_p)$ as diagonal subgroup and $\chi_1(a)=\chi\matrixx{a}{}{}{1}$.

Now we consider the situation that
\begin{itemize}
\item[1)] $F$ is a totally real field and $K$ is a totally imaginary quadratic extension over $F$,
\item[2)] for any place $v|\infty$ of $F$, $\pi_v^\JL$ is a unitary discrete series of weight $2$,
\item[3)] $(c, N)=1$.
\end{itemize}
Now let $\phi$ be the Hilbert newform as in Theorem \ref{GZ}, (which is different from the one we choose in Theorem \ref{W}). We are going to give an explicit form of Waldspurger formula following Gross \cite{Gross}, which is quoted in many references. Let $X=B^\times \bs \wh{B}^\times/\wh{R}^\times$ and let $g_1, \cdots, g_n \in \wh{B}^\times$ be a complete set of representatives of $X$. Denote $[g]\in X$ for the class of an element $g\in \wh{B}^\times$. Note that for each $g_i$, let $\Gamma_i=(B^\times \cap g_i \wh{R}^\times g_i^{-1})/\CO^\times$, which is finite and denote by $w_i$ its order. Let $\BZ[X]$ be the free $\BZ$-module (of rank $\# X$) of formal sums $\sum_i a_i [g_i]$.  There is a height pairing on $\BZ[X]\times \BZ[X]$ defined by
$$\pair{\sum a_i [g_i], \sum b_i[g_i]}=\sum_i a_i b_i w_i.$$
By Eichler's norm theorem, the norm map
$$\RN:\qquad  X:=B^\times \bs \wh{B}^\times /\wh{R}^\times\lra C_+:=F_+^\times \bs \wh{F}^\times /\wh{\CO}^\times$$
is surjective. For each $c\in C_+$, let $X_c\subset X$ be the preimage of $c$ and $\BZ[X_c]$ be the submodule of $\BZ[X]$ supported on $X_c$. Then $\BZ[X]=\bigoplus_{c\in C_+} \BZ[X_c]$. Let $\BZ[X_c]^0$ be the submodule of classes $\sum a_i [g_i]\in \BZ[X_c]$ with degree $\sum_i a_i=0$, and let $\BZ[X]^0=\bigoplus_{c\in C_+} \BZ[X_c]^0$ and $\BC[X]^0=\BZ[X]^0 \otimes_\BZ \BC$.  Note that $V(\pi, \chi)\subset \pi^{\wh{R}^\times}$ by Proposition \ref{local-multi-one-2}, and then there is an injection
$$V(\pi, \chi)\lra \BC[X]^0,\qquad f\mapsto \sum f([g_i]) w_i^{-1} [g_i],$$
and view $V(\pi, \chi)$ be a line on $\BC[X]^0$. It follows that $\pair{f, f}=\pair{f, f}_{\wh{R}^\times}$. The fixed embedding $K\ra B$ induces a map
$$\Pic(\CO_{c}) \lra X, \qquad t\mapsto x_t,$$
using which we define an element in $\BC[X]$,
$$P_\chi:=\sum_{t\in \Pic(\CO_{c})} \chi^{-1}(t) x_t$$
and let $P_\chi^\pi$ be its projection to the line $V(\pi, \chi)$. Then the Explicit formula in Theorem \ref{W} implies
\begin{thm} \label{G} Let $(\pi, \chi)$ be as above with conditions 1)-3). The height of $P_\chi^\pi$ is given by the formula
 $$L^{(\Sigma)}(1/2,\pi,\chi) =
    2^{-\#\Sigma_D}\cdot \frac{(8\pi^2)^d\cdot (\phi, \phi)_{U_0(N)}}{u^2\sqrt{|D_K| \|c\|^{2}}}
    \cdot \pair{P_\chi^\pi, P_\chi^\pi},$$
where   $$\Sigma:=\left\{v|(N, D)\infty, \text{if $v\|N$ then $v\nmid D$}\right\},\qquad
\Sigma_D := \left\{ v| (N,D)   \right\},$$
 $u=\#\kappa_c\cdot [\CO_{c}^\times:\CO^\times]$, and $\phi\in \pi^\JL$ is the Hilbert newform as in Theorem \ref{GZ}.
For any non-zero vector $f\in V(\pi, \chi)$, let $P_\chi^0(f)=\sum_{t\in \Pic(\CO_{c})} f(t) \chi(t)$, then we have $$L^{(\Sigma)}(1/2,\pi,\chi) =
    2^{-\#\Sigma_D}\cdot \frac{(8\pi^2)^d\cdot (\phi,\phi)_{U_0(N)}}{u^2\sqrt{|D_K| \|c\|^{2}}}
    \cdot \frac{|P_\chi^0(f)|^2}{\pair{f, f}}.$$
\end{thm}

\s{\bf Remark}. When $c$ and $N$ have common factor, one can still formulate an explicit formula in the spirit of Gross by defining a system of height pairings $\pair{\ ,\ }_U$ in the same way of Theorem \ref{W}.

As a byproduct, we obtain the following theorem about the relation between Petersson norm of newform and a special value of adjoint L-function.
\begin{prop}Let $F$ be a totally real field and $\sigma$ a cuspidal unitary automorphic representation of $\GL_2(\BA)$ of conductor $N$ such that for any $v|\infty$, $\sigma_v$ is a discrete series of weight $k_v$.  Let $\phi$ be the Hilbert newform in $\sigma$ as in Theorem \ref{GZ}. Then we have that
$$\frac{L^{(S)}(1, \sigma, \ad)}{(\phi, \phi)_{U_0(N)}}=2^{d-1+\sum_{v|\infty} k_v}\cdot\|N\delta^{-2}\|^{-1} \cdot h_F^{-1},$$
where $S$ is the set of finite places $v$ of $F$
with $\ord_v(N)\geq 2$, $h_F$ is the ideal class number of $F$, and
$$(\phi, \phi)_{U_0(N)}=\iint_{X_{U_0(N)}} |\phi|^2 \left(\bigwedge_{v|\infty} y_v^{k_v-2} dx_v dy_v\right), \qquad z_v=x_v+y_v i.$$ Or equivalently,
$$\frac{L^{(N\infty)}(1, \sigma, \ad)}{(\phi, \phi)_{U_0(N)}}=\frac{1}{2}\cdot \| N\delta^{-2}\|^{-1} \cdot h_F^{-1} \cdot \prod_{v|\infty} \frac{4^{k_v} \pi^{k_v+1}}{\Gamma(k_v)} \cdot \prod_{v\| N} L(2, 1_{F_v})^{-1}.$$
\end{prop}
\begin{proof}
This follows from Proposition \ref{Petersson}, Lemma \ref{volume},  and  Proposition \ref{local inner product}.

\end{proof}
\s{\bf Example}. Assume that $F=\BQ$ and $\sigma$ is the cuspidal automorphic representation associated to a
cuspidal newform $\phi\in S_k(\SL_2(\BZ))$. Then we have that
$$L(1, \sigma, \ad)=2^k \cdot (\phi, \phi)_{\SL_2(\BZ)}, \quad\qquad  L^{(\infty)}(1, \sigma, \ad)=\frac{2^{2k-1}\pi^{k+1}}{\Gamma(k)} \cdot (\phi, \phi)_{\SL_2(\BZ)}.$$

\section{Reduction to Local Theory}
We now explain how to obtain these explicit formulas in Theorems \ref{GZ} and \ref{W} from the original Waldspurger formula and the general Gross-Zagier formula proved by Yuan-Zhang-Zhang in \cite{YZZ}. We first consider  Waldspurger formula. Let $B$ be a quaternion algebra over a number field $F$ and $\pi$ a cuspidal automorphic representation on $B_\BA^\times$ with central character $\omega$. Let $K$ be a quadratic field extension over $F$ and $\chi$ be a Hecke character on $K_\BA^\times$. Assume that (1) $\omega\cdot \chi|_{\BA^\times}=1$, (2) for any place $v$ of $F$, $\epsilon(1/2, \pi_v, \chi_v)=\chi_v\eta_v(-1) \epsilon(B_v)$. Define Petersson pairing on $\pi\otimes\wt{\pi}$ by
 $$\pair{f_1, f_2}_\Pet=\int_{B^\times \BA^\times \bs B_\BA^\times} f_1(g) f_2(g) dg$$
 with Tamagawa measure so that the volume of $B^\times \BA^\times \bs B_\BA^\times$ is $2$. Let $P_\chi$ denote the period functional on $\pi$:
$$P_\chi(f)=\int_{K^\times \BA^\times\bs K_\BA^\times}f(t)\chi(t) dt, \qquad \forall f\in \pi.$$
Then Waldspurger's period formula (\cite{Wal} or Theorem 1.4 of \cite{YZZ}) says that for any pure tensors $f_1\in \pi, f_2\in \wt{\pi}$ with $\pair{f_1, f_2}_\Pet\neq 0$,

\begin{equation}\label{1}\frac{P_\chi(f_1) P_{\chi^{-1}}(f_2)}{\pair{f_1, f_2}_\Pet}=\frac{L(1/2, \pi, \chi)}{2L(1,\pi, \ad) L(2, 1_F)^{-1}}\cdot \prod_v \beta_v(f_{1,v}, f_{2, v}),\end{equation}
where $L(1, \pi, \ad)$ is defined using the Jacquet-Langlands lifting of $\pi$, and  for any place $v$ of $F$, let $\pair{\ ,\ }_v:  \pi_v \times \wt{\pi}_v\ra \BC$ be a non-trivial  invariant pairing, and
$$\beta(f_{1, v}, f_{2, v})=\frac{L(1, \eta_v) L(1, \pi_v, \ad)}{L(1/2, \pi_v, \chi_v) L(2, 1_{F_v})} \int_{K_v^\times/F_v^\times} \frac{\pair{\pi(t_v) f_{1, v}, f_{2, v}}_v}{\pair{f_{1, v}, f_{2, v}}_v} \chi(t_v) dt_v,$$
Here local  Haar measures $dt_v$ are chosen such that
$\otimes_v dt_v=dt$ is the Haar measure on $K_\BA^\times/\BA^\times$ in the definitions of $P_\chi$ and $P_{\chi^{-1}}$, and  the volume of $K^\times \bs K_\BA^\times/\BA^\times $ with respect to $dt$ is $2L(1, \eta)$. Note that the Haar measure $dt$ is different from the one used in  Theorem 1.4 of \cite{YZZ}.  To obtain the explicit formula, we first relate $P_\chi(f), L(1, \pi, \ad)$, and $\pair{f_1, f_2}_\Pet$ to the corresponding objects with levels in Theorem \ref{W}, and reduce to local computation.

For our purpose, it is more convenient to normalize local additive characters and local Haar measures as follows. Take the additive character $\psi = \otimes_v \psi_v$  on $\BA$ as follows:
\[\psi_v(a) =
  \begin{cases}
    e^{2\pi ia}, \quad &\text{if } F_v = \BR; \\
    e^{4\pi i\Re(a)}, \quad &\text{if } F_v = \BC; \\
    \psi_p(\tr_{F/\BQ_p}(a)), \quad &\text{if $F_v$ is a finite extension over $\BQ_p$ for some prime $p$,}
  \end{cases}\]
where $\psi_p(b) = e^{-2\pi i \iota(b)}$ and $\iota: \BQ_p/\BZ_p \ra \BQ/\BZ$ is the natural embedding. It turns out that $\psi$ is a character on $F\bs \BA$.
 For any place $v$  of $F$, let $da_v$ denote the Haar measure on $F_v$ self-dual to $\psi_v$ and let $d^\times a_v$ denote
the Haar measure on $F_v^\times$ defined by $d^\times a_v  = L(1,1_v)\frac{da_v}{|a_v|_v}$.
Let $L$ be  a separable quadratic extension of $F_v$ or
a quaternion algebra over $F_v$ and $q$  the reduced norm on $L$,  then $(L, q)$ is a quadratic space over $F_v$. Fix the Haar measure $dx$ on $L$ to be the one self-dual with respect to $\psi_v$ and  $q$ in the sense that $\wh{\wh{\Phi}}(x)=\Phi(-x)$ for any $\Phi\in S(L)$, where $\wh{\Phi}(y):=\int_L \Phi(x) \psi_v(\pair{x, y}) dx$ is the Fourier transform of $\Phi$ and $\pair{x, y}=q(x+y)-q(x)-q(y)$ is the bilinear form on $L$ associated to $q$.  Fix the Haar measure $d^\times x$
on $L^\times$ to be the one defined by
\[d^\times x=
  \begin{cases}
    \displaystyle{L(1,1_v)^2\frac{d x}{|q(x)|_v}}, \quad &\text{if } L = F_v^2; \\
    \displaystyle{L(1,1_L)\frac{d x}{|q(x)|_v}}, \quad &\text{if $L$ is a quadratic field extension over $F_v$}; \\
   \displaystyle{L(1,1_v)\frac{d x}{|q(x)|_v^2}}, \quad &\text{if $L$ is a quaternion algebra}.
\end{cases}\]
Endow $L^\times/F_v^\times$ with the quotient Haar measure.  Let $K$ be a quadratic field extension over $F$ and $B$ a quaternion algebra over $F$. For local Haar measures on $K_v^\times /F_v^\times$ and $B_v^\times/F_v^\times$,  their product Haar measures on
$K_\BA^\times/\BA^\times$ and $B_\BA^\times/\BA^\times$ satisfy the following:
$$\Vol( K^\times\bs K_\BA^\times/\BA^\times) = 2L(1,\eta), \qquad \quad \Vol(B^\times\bs B_\BA^\times/\BA^\times ) = 2.$$
Thus these measures can be taken as the ones used in the above statement of Waldspurger's formula. From now on, we always use these measures and the additive character $\psi$ on $\BA$.

\subsection{Petersson Pairing Formula}
Let $\sigma$ be a cuspidal automorphic representation of $\GL_2(\BA)$ and $\wt{\sigma}$ its contragredient. Let $N$ be the unipotent subgroup $N=\left\{\matrixx{1}{x}{}{1}, x\in F\right\}$ of $\GL_2$.
View $\psi$ as a character on $N(F)\bs N(\BA)$ and the Haar measure $da$ on $\BA$ as the one on $N(\BA)$.  For any $\phi\in \sigma$, let $W_\phi\in \CW(\sigma, \psi)$ be the Whittaker function associated to $\phi$:
$$W_\phi(g):=\int_{N(F)\bs N(\BA)} \phi(n g) \ov{\psi(n)} dn.$$ Recall there is a $\GL_2(F_v)$-pairing on $\CW_{\sigma_v, \psi_v}\times \CW_{\wt{\sigma}_v, \psi_v^{-1}}$: for any local Whittaker functions $W_{1, v}\in \CW(\sigma_v, \psi_v), W_{2, v}\in \CW(\wt{\sigma}_v, \psi_v^{-1})$,
$$\pair{W_{1, v}, W_{2, v}}=\int_{F_v^\times} W_{1, v}\matrixx{a}{}{}{1} W_{2, v} \matrixx{a}{}{}{1} d^\times a.$$
Define the Petersson pairing on $\sigma \times \wt{\sigma}$ by:
$$\pair{\phi_1, \phi_2}_\Pet:=\int_{Z(\BA) \GL_2(F)\bs \GL_2(\BA)} \phi_1(g) \phi_2(g) dg, \qquad \phi_1\in \sigma, \phi_2\in \wt{\sigma},$$
where $Z\cong F^\times$ is the center of $\GL_2$.

\begin{prop}\label{Petersson} For any $\phi_1\in \sigma, \phi_2\in \wt{\sigma}$ pure tensors and write $W_{\phi_i}=\otimes_v W_{i, v}$, $i=1, 2$, we have
  \begin{equation}\label{22}\pair{\phi_1, \phi_2}_\Pet=2 L(1, \sigma, \ad) L(2, 1_F)^{-1} \prod_v \alpha_v(W_{1, v}, W_{2, v}),\end{equation}
where for any place $v$ of $F$
$$\alpha_v(W_{1, v}, W_{2, v})=\frac{1}{L(1, \sigma_v, \ad)L(1, 1_v)L(2, 1_v)^{-1}}\cdot \pair{W_{1, v}, W_{2, v}},$$
\end{prop}

\begin{proof}We may assume that the cuspidal automorphic representation $\sigma$ is also unitary and identify $\wt{\sigma}$ with $\ov{\sigma}$.  Let $G=\GL_2$ over $F$, $P$ the parabolic subgroup of upper triangle matrices in $G$, and let $U=\prod_v U_v$ be a maximal compact subgroup of $G(\BA)$.  For any place $v$ of $F$, with respect to the Iwasawa
decomposition of $G(F_v)$,
$$g=a\matrixx{1}{x}{}{1}\matrixx{1}{}{}{b} k\in
G(F_v),\qquad  a, b\in F_v^\times, x\in F_v, k\in U_v,$$
choose the measure $dk$ on $U_v$ such that  $dg=|b|dx d^\times a d^\times b dk$ is the fixed local Haar measure on $G(F_v)$. Note that for $v$  non-archimedean, $U_v$ has volume $L(2, 1_v)^{-1}|\delta_v|^{1/2}$ with respect to $dk$ and has volume $L(2, 1_v)^{-1}|\delta_v|^2$ with respect to  the fixed measure on $G(F_v)$; for $v$ archimedean, $U_v$ has volume $L(2, 1_v)^{-1}$ with respect to $dk$.

By Lemma 2.3. in \cite{JC}, for any Bruhat-Schwartz function $\Phi_v\in \CS(F_v^2)$, we have
$$\int_{F_v^\times \times U_v} \Phi([0, b]k) |b|^2 d^\times b dk
=\wh{\Phi}_v (0),$$where $\wh{\Phi}_v$ is the Fourier transformation of $\Phi_v$ and $\wh{\Phi}_v(0)$ is independent of the choice of the additive character $\psi_v$. For any $\Phi\in \CS(\BA^2)$, let
$$F(s, g, \Phi)=|\det g|^s \int_{\BA^\times} \Phi([0, b]g)|b|^{2s}
d^\times b,$$and define the Eisenstein series
$$E(s, g, \Phi):=\sum_{\gamma \in P(F)\bs G(F)} F(s, \gamma g,
\Phi), \quad \Re(s)\gg 0.$$By Possion summation formula,
$$\begin{aligned}
E(s, g, \Phi)=&|\det g|^s \int_{F^\times \bs \BA^\times}
  \left(\sum_{\xi\in
F^2\setminus \{0\}} \Phi(a\xi g)\right) |a|^{2s}d^\times a\\
=&|\det g|^s\int_{|a|\geq 1}
\left(\sum_{\xi\in F^2\setminus \{0\}} \Phi(a\xi g)\right) |a|^{2s}d^\times a\\
&+|\det g|^{s-1} \int_{|a|\geq 1}\left(\sum_{\xi\in F^2\setminus
\{0\}} \wh{\Phi}(g^{-1}\xi {}^t a)\right) |a|^{2-2s}d^\times
a\\
&+|\det g|^{s-1}\wh{\Phi}(0)\int_{|a|\leq 1} |a|^{2s-2} d^\times a\\
&-|\det g|^s\Phi(0)\int_{|a|\leq 1} |a|^{2s}d^\times a.
\end{aligned}$$
It shows that $E(s, g, \Phi)$ has meromorphic continuation to the whole $s$-plane,  has
only possible poles at $s=0$ and $1$,  and its residue at $s=1$ is equal to
$$\Res_{s=1}E(s, g, \Phi)=\wh{\Phi}(0) \lim_{s\ra 1} (s-1)\int_{|a|\leq 1}
|a|^{2s-2} d^\times a=\displaystyle{\frac{\wh{\Phi}(0)}{2}\Res_{s=1} L(s, 1_F)},$$which is independent of $g$.  By unfolding the Eisenstein series and
Fourier expansions of $\phi_i$, the integral
$$\begin{aligned}
Z(s, W_{\phi_1}, W_{\phi_2}, \Phi)&=Z(s, \phi_1, \phi_2, \Phi):=
\int_{[Z\bs G]} \phi_1(g) \phi_2(g) E(s, g, \Phi)dg\\
&=\int_{N(\BA)\bs G(\BA)} |\det g|^s W_{\phi_1}(g) W_{\phi_2}(g)\Phi([0, 1]g)
dg \end{aligned}$$
has an Euler product if $\Phi\in S(\BA^2)$ is a pure tensor. For each place $v$ of $F$
and $\Phi_v\in S(F_v^2)$, denote
$$Z(s, W_{1, v}, W_{2, v}, \Phi_v)
=\int_{N(F_v)\bs G(F_v)} |\det g|^s W_{1, v}(g) W_{2, v}(g)\Phi_v([0,
1]g) dg,$$which has meromorphic continuation to the whole $s$-plane,
and moreover, for $v\nmid \infty$, the fractional ideal of
$\BC[q_v^s, q_v^{-s}]$ of all $Z(s, W_{1, v}, W_{2, v}, \Phi_v)$
with $W_{1, v}\in \CW(\sigma_v, \psi_v)$, $W_{2, v}\in \CW(\wt{\sigma}_v, \psi_v^{-1})$,  and $\Phi_v \in \CS(F_v^2)$ is
generated by $L(s, \sigma_v\times \wt{\sigma}_v)$. It is also known that for each $v$,
$$Z(1, W_{1, v}, W_{2, v}, \Phi_v)=\int_{F^\times} W_{1, v}\matrixx{a}{}{}{1}W_{2,
v}\matrixx{a}{}{}{1}d^\times a\cdot \iint_{F_v^\times\times U_v}
\Phi_v([0, b]k) |b|^2 d^\times b dk,$$ with Haar measures
 chosen above.   Let
$\Phi=\otimes_v \Phi_v\in \CS(\BA^2)$ be a pure tensor such that $\wh{\Phi}(0)\neq 0$  and take residue at $s=1$ on the two sides of
$$Z(s, \phi_1, \phi_2, \Phi)=\prod_v Z(s, W_{1, v}, W_{2, v}, \Phi_v),$$
we have
$$\pair{\phi_1, \phi_2}_\Pet \Res_{s=1} E(s, g, \Phi)=\Res_{s=1}L(s, \sigma\times
\wt{\sigma})\wh{\Phi}(0) \prod_v\frac{\pair{W_{1, v}, W_{2, v}}}{L(1,
\sigma_v \times \wt{\sigma}_v)},$$or
$$\frac{L(1, \sigma, \ad)}{\pair{\phi_1, \phi_2}_\Pet}=\frac{1}{2}\prod_v
\frac{L(1, \sigma_v, \ad)L(1, 1_{F_v})}{\pair{W_{1, v}, W_{2, v}}}.$$
The formula in the Proposition follows.

\end{proof}

\subsection{$U$-level Pairing}

\begin{lem}\label{volume} Let $B$ be a quaternion algebra  over a number field $F$ and denote  by  $r, s, t$  integers such that $B\otimes_\BQ \BR\cong \BH^r\times M_2(\BR)^s \times M_2(\BC)^t$. For $U\subset \wh{B}^\times$ an open compact subgroup,  the volume of $X_U$, defined after Definition \ref{test space},  is given by
$$\Vol(X_U)=2(4\pi^2)^{-d} \# (\BA_f^\times/F^\times U_Z)\cdot \frac{\Vol(U_Z)}{\Vol(U)}.$$
where $U_Z=U\cap \wh{F}^\times$ and the volumes $\Vol(U_Z)$ and $\Vol(U)$ are with respect to Tamagawa measure so that $\Vol(\GL_2(\CO_v))=L(2, 1_v)^{-1} \Vol(\CO_v)^4$, and  $\Vol(B_v^\times)=L(2, 1_v)^{-1}\Vol(\CO_v)^4 (q_v-1)^{-1}$ for $B_v$ division. In particular, if $U$ contains $\wh{\CO}^\times$, then
$$\Vol(X_U)=2(4\pi^2)^{-d} |D_F|^{-1/2}\cdot h_F \cdot \Vol(U)^{-1},$$
where $h_F$ is the class number of $F$.
\end{lem}
\begin{proof}(See also \cite{YZZ} for the case $s=1$ and $t=0$). Denote by $B^1 := \left\{b \in B^\times | q(b) = 1  \right\}$ with $q$ the reduced norm on $B$. For each place
$v$ of $F$, we have the exact sequence
$$1\lra B^1_v \lra B_v^\times \lra q(B_v^\times)\lra 1,$$
and  define the Haar measure $dh_v$ on $B_v^1$ such that it makes the Haar measure on $q(B_v^\times)$, obtained by the restriction of the Haar measure on $F_v^\times$, equal to the
quotient of the Haar measure on $B_v^\times$ by $dh_v$. The product of these local measures give the Tamagawa measure on
$B^1_\BA$ so that $\Vol(B^1\bs B^1_\BA)=1$.
This follows from the fact that the Tamagawa numbers of $B^1$ and $B^\times$ are $1$ and $2$, respectively. Assume that $B\otimes_\BQ \BR=\BH^r\times M_2(\BR)^s\times M_2(\BC)^t$. We assume that $s+t>0$ first and let $\Sigma \subset \infty$ be the subset of infinte places of $F$ where $B$ splits. By the strong approximation theorem, $B_\BA^1=B^1 B_\infty^1 U^1$, where $U^1=U\cap B^1_{\BA_f}$  is an open compact subgroup of $B^1_{\BA_f}$. It follows that
$$B^1\bs B_\BA^1=B^1\bs B^1 B_\infty^1 U^1=(\Gamma \bs B_{\Sigma}^1) B_\infty^{1, \Sigma} U^1$$
where $\Gamma=B^1\cap U^1$ and we identify $\Gamma \bs B^1_{\Sigma}$ with the fundamental domain of this quotient.

For a real place $v$ of $F$, $B_v^1\cong \SL_2(\BR)$. By the Iwasawa decomposition, any element is uniquely of the form
$$\matrixx{1}{x}{}{1} \matrixx{y^{1/2}}{}{}{y^{-1/2}} \matrixx{\cos\theta}{\sin\theta}{-\sin\theta}{\cos\theta}, \qquad x\in \BR, y\in \BR_+, \theta\in [0, 2\pi).$$
The measure on $B_v^1$ is $dx dy d\theta/2y^2$ with $dx dy$ the usual Lebesgue and $\theta$ has volume $2\pi$. For a complex place $v$ of $F$, $B_v^1\cong \SL_2(\BC)$. By the Iwasawa decomposition, any element in $\SL_2(\BC)$ is uniquely of form
$$\matrixx{1}{z}{}{1}\matrixx{v^{1/2}}{}{}{v^{-1/2}}u, \qquad z\in \BC, v\in \BR_+, u\in \SU_2,$$
The measure on $B_v^1$ is $dxdydv du/v^3$ with $z=x+yi$, $dx, dy, dv$ the usual Lebesgue measure, and $du$ has volume $8\pi^2$ (see \cite{Vig}). It follows that
$$\Vol(\Gamma\bs B^1_\Sigma)=2^{-t}(4\pi^2)^{s+2t} w_U^{-1}\cdot \Vol\left(\Gamma\bs (\CH_2^s\times \CH_3^t), \ \wedge \frac{dxdy}{4\pi y^2} \wedge  \frac{dx dy dv}{\pi^2v^3}\right),$$where $w_U=\# \{\pm 1\}\cap U$.
On the other hand, for any infinite place $v\notin \Sigma$, $\Vol(B_v^1)=4\pi^2$. It follows that
$$w_U^{-1}\cdot 2^{-t}(4\pi^2)^d \cdot \Vol\left(\Gamma\bs (\CH_2^s\times \CH_3^t), \ \wedge \frac{dxdy}{4\pi y^2} \wedge  \frac{dx dy dv}{\pi^2 v^3}\right)\cdot \Vol(U^1)=1,$$
where $d=[F:\BQ]$. Let $B_+^\times\subset B^\times$ be the subgroup of elements whose norms are positive at all real places. Now consider the natural map
$$(B^1\cap U^1) \bs (\CH_2^s\times \CH_3^t)\lra  (B_+^\times \cap U) \bs (\CH_2^s\times \CH_3^t),$$
whose degree is just
$$[(B_+^\times \cap U): (B^1\cap u^1)\mu_U]=[\det (B_+^\times \cap U): \mu_U^2]=[\mu_U': \mu_U^2].$$
Here $\mu_U=F^\times \cap U$  and $\mu_U'=F_+^\times\cap \det U$, subgroups of $\CO_F^\times$ with finite index, . It follows that
$$\begin{aligned}
\Vol(X_U)=& \Vol((B_+^\times\cap U)\bs (\CH_2^s\times \CH_3^t)\cdot \#(F_+^\times \bs \wh{F}^\times /\det U)\\
&=\frac{2^tw_U}{(4\pi^2)^{d} \cdot \Vol(U^1)\cdot [\mu_U': \mu_U^2]} \cdot \#(F_+^\times \bs \wh{F}^\times /\det U).
\end{aligned}.$$
Note that
$$\frac{\#(\wh{F}^\times/F_+^\times \det(U))}{\#(\wh{F}^\times \bs
F^\times U_Z))} =[F^\times U_Z: F_+^\times \det U]=[F^\times:
F_+^\times]\frac{\Vol(U_Z)}{\Vol(\det (U))} [\mu_U':\mu_U].$$Note that $[F^\times: F_+^\times]=2^{r+s}$,
$[\mu_U:\mu_U^2]=2^{r+s+t-1}w_U$, and $\Vol(U)=\Vol(U^1)\Vol(\det (U))$, we have
$$\Vol(X_U)=2(4\pi^2)^{-d} \# (\wh{F}^\times/F^\times U_Z)\cdot \frac{\Vol(U_Z)}{\Vol(U)}.$$

Now assume that $s=t=0$. Note that the Tamagawa number of $B^\times$ is $2$, $\Vol(B_v^\times/F_v^\times)=4\pi^2$ for any $v|\infty$,  and the decomposition
$$B^\times \BA^\times \bs B^\times_\BA =F_\infty^\times \bs B_\infty^\times \times B^\times \wh{F}^\times \bs \wh{B}^\times.$$
It follows that $\Vol(B^\times \wh{F}^\times \bs \wh{B}^\times)=2 (4\pi^2)^{-d}$.  Let $\gamma_1, \cdots, \gamma_h$ be a complete set of representatives in $\wh{B}^\times$ of the coset $B^\times \bs \wh{B}^\times/U$. Consider the natural map
$$B^\times \bs B^\times \gamma_i U \lra B^\times \wh{F}^\times \bs B^\times \wh{F}^\times \gamma_iU,$$
 whose degree is $\#\wh{F}^\times/F^\times U_Z$. Since  $$\Vol(B^\times \wh{F}^\times \bs B^\times \wh{F}^\times \gamma_iU)=\Vol\left(\frac{\gamma_i(U/U_Z)\gamma_i^{-1}}{(B^\times \cap \gamma_i U \gamma_i^{-1})/\mu_Z } \right)=\frac{\Vol(U)/\Vol(U_Z)}{\# (B^\times \cap \gamma_i U \gamma_i^{-1})/\mu_Z }.$$
Thus
$$2(4\pi^2)^{-d}=\Vol(B^\times \wh{F}^\times \bs \wh{B}^\times)=(\#\wh{F}^\times/F^\times U_Z)^{-1}\cdot \frac{\Vol(U)}{\Vol(U_Z)}\cdot \sum_{i=1}^h \frac{1}{\# (B^\times \cap \gamma_i U \gamma_i^{-1})/\mu_Z}.$$

\end{proof}

\subsection{$c_1$-level Periods}
Now take $f_1\in V(\pi, \chi). f_2\in V(\wt{\pi}, \chi^{-1})$ to be non-zero test vectors defined before. Let $\sigma=\pi^\JL$, take $\phi_1\in \sigma$ and $\phi_2\in \wt{\sigma}$ to be normalized new vectors. The $c_1$-level periods  $P^0_\chi(f_1), P^0_{\chi^{-1}}(f_2)$ are related to the periods in Waldspurger's formula by the following  lemma:

\begin{lem}\label{period} Let $b\subset \CO$ be a non-zero ideal of $F$ and denote by $\Pic_{K/F}(\CO_b)$ the group $\wh{K}^\times/K^\times \wh{F}^\times \wh{\CO}_b^\times$. Then there is a relative class number formula:
$$L^{(b)}(1, \eta)\cdot \|D_{K/F}b^2\delta\|^{1/2}\cdot 2^{-r_{K/F}}=\frac{\# \Pic_{K/F}(\CO_b)\cdot R_b}{\#\kappa_b\cdot w_b},$$
where $r_{K/F}=\rank \CO_K^\times -\rank \CO^\times$, $w_b=[\CO^\times_{b, \tor}: \CO^\times_\tor]$, $R_b$ is the quotient of the regulator of $\CO_b^\times$ by the one of $\CO^\times$, and $\kappa_b$ is  the kernel of the natural morphism from $\Pic(\CO)$ to $\Pic(\CO_b)$.   Define a constant $\nu_b: =2^{-r_{K/F}} R_b^{-1}\cdot \#\kappa_b w_b$.
Then it follows that
$$P_\chi(f)=2L_{c_1}(1, \eta) \|Dc_1^2\delta\|^{-1/2} \nu_{c_1}^{-1} \cdot P_\chi^0(f).$$
\end{lem}

\begin{proof} There are exact sequences
$$1\ra \kappa_b \ra \wh{F}^\times/F^\times \wh{\CO}_F^\times\ra  \wh{K}^\times/K^\times
\wh{\CO}_b^\times \ra \wh{K}^\times/K^\times
\wh{F}^\times \wh{\CO}_b^\times \ra 1,$$
and
$$1\ra \CO_K^\times/\CO_b^\times \ra \wh{\CO}_K^\times/\wh{\CO}_b^\times \ra \wh{K}^\times/K^\times \wh{\CO}_b^\times \ra \wh{K}^\times/K^\times \wh{\CO}_K^\times \ra 1.$$

It follows that
$$\#\Pic_{K/F}(\CO_b)=\# \wh{K}^\times/K^\times \wh{F}^\times \wh{\CO}_b^\times=\frac{h_K}{h_F}\cdot
[\wh{\CO}_K^\times:\wh{\CO}_b^\times]\cdot
[\CO_K^\times:\CO_b^\times]^{-1}\cdot \#\kappa_b,$$ where $h_K=\#\wh{K}^\times/K^\times\wh{\CO}_K^\times$ is the ideal
class number of $K$ and similarly for $h_F$. By class number formula for $F$ and $K$, we have
$$\Res_{s=1}L(s, 1_F)=2^{r_F+1}\frac{R_F h_F}{w_F \sqrt{|D_F|}}, \quad \Res_{s=1}L(s, 1_K)=2^{r_K+1}\frac{R_K h_K}{w_K \sqrt{|D_K|}}$$
where $r_F=\rank\CO_F^\times$, $D_F$ is the discriminant of $F$, $R_F$ is the regulator of $\CO^\times$,  $h_F$ the ideal class number of $F$,  $w_F=\# \CO_{\tor}^\times$, and similar for $r_K, D_K, R_K, h_F$ and $w_K$. Noting that $|D_K|/|D_F|=|D_{K/F}\delta|_\BA^{-1}$ and $[\wh{\CO}_K^\times: \wh{\CO}_b^\times]^{-1}=L_b(1, \eta)|b|$, we have that
$$\begin{aligned}
L(1, \eta)&=\frac{h_K}{h_F}\cdot 2^{r_{K/F}} \frac{R_Kw_K^{-1}}{R_Fw_F^{-1}}\cdot \|D_{K/F} \delta\|^{-1/2}\\
&=\#\Pic_{K/F}(\CO_b)\cdot L_b(1, \eta) \cdot 2^{r_{K/F}}  \cdot [\CO_K^\times: \CO_b^\times]\frac{R_Kw_K^{-1}}{R_Fw_F^{-1}}(\#\kappa_b)^{-1}\cdot \|D_{K/F} b^2\delta\|^{-1/2}.\end{aligned}$$
 The relative class number formula then follows.
\end{proof}

Let $N$ be the conductor of $\sigma=\pi^\JL$, let $U\subset \wh{B}^\times$ be an open compact subgroup, recall
$$\langle f_1, f_2\rangle_U=\frac{\pair{f_1, f_2}_\Pet}{2} \Vol(X_U), \qquad
\langle\phi_1, \phi_2\rangle_{U_0(N)}=\frac{\pair{\phi_1, \phi_2}_\Pet}{2} \Vol(X_{U_0(N)}).$$
Applying Proposition \ref{Petersson}, Lemma \ref{volume}, and Lemma \ref{period},   Waldspurger's formula \eqref{1} implies the following.
\begin{prop}\label{B1} Let $U=\prod_v U_v\subset \wh{B}^\times$ be an open compact subgroup such that $\wh{\CO}^\times \subset U$. Denote by $\gamma_v=\Vol(U_0(N)_v)^{-1}\Vol(U_v)$ for all finite places $v$ and $\gamma_v=1$ for $v|\infty$. Let $\phi_1\in \pi^\JL, \phi_2\in \wt{\pi}^\JL$ be any forms with $\pair{ \phi_1, \phi_2}_{U_0(N)}\neq 0$ and let $\alpha (W_{1, v}, W_{2, v})$ be corresponding local constants defined in Proposition \ref{Petersson}. Let $f_1\in \pi, f_2\in \wt{\pi}$ be any pure tensors with $(f_1, f_2)_\Pet\neq 0$ and $\beta(f_{1, v}, f_{2, v})$ corresponding constants defined in \eqref{1}.  Then we have
\begin{equation}\label{final reduction}
(2L_{c_1}(1, \eta) |Dc_1^2\delta|_\BA^{1/2}\nu_{c_1}^{-1})^2\cdot
\frac{P^0_\chi(f_1) P^0_{\chi^{-1}}(f_2)}{\langle f_1, f_2 \rangle_U}
=\frac{L(1/2, \pi, \chi)}{ \langle \phi_1, \phi_2\rangle_{U_0(N)}}\cdot  \prod_{v} \alpha_v(W_{1, v}, W_{2, v})\beta_v(f_{1, v}, f_{2, v})\gamma_v,\end{equation}
where $\nu_{c_1}$ is defined as in Lemma \ref{period}.
\end{prop}
It is now clear that the explicit Waldspurger formula will follow from the computation  of these local factors. In the next section, we will choose $\phi_1, \phi_2$ to be normalized new vectors in $\pi^\JL$ and $\wt{\pi}^\JL$, respectively, choose non-zero $f_1\in V(\pi, \chi), f_2\in V(\wt{\pi}, \chi)$, and compute the related local factors in \eqref{final reduction}.

We obtain Explicit Gross-Zagier formula from Yuan-Zhang-Zhang's formula via a similar way. Let $F$ be a totally real field and $X$ a Shimura curve over $F$ associated to an incoherent quaternion algebra $\BB$. Let $A$ be an abelian variety over $F$ parametrized by $X$ and let $\pi_A=\Hom_\xi^0(X, A)$ be the associated automorphic representation of $\BB^\times$ over the field $M:=\End^0(A)$ and $\omega$ its central character.  Let $K$ be a totally imaginary quadratic extension over $F$ and $\chi: K_\BA^\times \ra L^\times$ a finite order Hecke character over a finite extension $L$ of $M$ such that $\omega\cdot \chi|_{\BA^\times}=1$ and for all places $v$ of $F$,
$\epsilon(1/2, \pi_A, \chi)=\chi_v\eta_v(-1)\epsilon(\BB_v)$.  Fix an embedding $K_\BA\ra \BB$ and $K_\BA^\times\ra \BB^\times$, let $P\in X^{K^\times}(K^\ab)$ and define
$$P_\chi(f)=\int_{K_\BA^\times/K^\times \BA^\times} f(P)^{\sigma_t}\otimes_M \chi(t) dt \in A(K^\ab)_\BQ \otimes_M L,$$
where we use the Haar measure so that the
totally volume of $K_\BA^\times/K^\times \BA^\times$ is $2L(1, \eta)$, and $\eta$ is the quadratic Hecke character on $\BA^\times$ associated to the extension $K/F$.  We further assume for all non-archimedean places $v$ that the compact subgroup $\CO_{K_v}^\times/\CO_v^\times$ has a volume in $\BQ^\times$ and fix a local invariant pairing $(\ ,\ )_v$ on $\pi_{A, v}\times \pi_{A^\vee, v}$ valued in $M$.  Define $\beta(f_{1, v}, f_{2, v})\in L$, with $(f_{1, v}, f_{2, v})_v\neq 0$, by
$$\beta(f_{1, v}, f_{2, v})=\frac{L(1, \eta_v) L(1, \pi_v, \ad)}{L(1/2, \pi_v, \chi_v) L(2, 1_{F_v})} \int_{K_v^\times/F_v^\times} \frac{(\pi(t_v) f_{1, v}, f_{2, v})_v}{(f_{1, v}, f_{2, v})_v} \chi(t_v) dt_v\in L.$$
where we take an embedding of $L$ into $\BC$ and  the above integral lies in fact in $L$ and does not depend on the embedding.

Then for any pure tensors  $f_1\in \pi_A, f_2\in\pi_{A^\vee}$ with $(f_1, f_2)\neq 0$,  Yuan-Zhang-Zhang obtained the following celebrated formula in \cite{YZZ} as an identity in $L\otimes_\BQ \BC$:
\begin{equation} \label{2} \frac{\pair{P_\chi(f_1), P_{\chi^{-1}}(f_2)}_{K, L}}{\Vol(X_U)^{-1}(f_1, f_2)_U}=\frac{L'(1/2, \pi_A, \chi)}{L(1, \pi_A, \ad) L(2, 1_F)^{-1}}\prod_v \beta_v(f_{1, v}, f_{2, v}).\end{equation}

Note that we use height over $K$ and the one used in \cite{YZZ} is over $F$, the Haar measure to define $P_\chi(f)$ is different from the one in \cite{YZZ} by $2L(1, \eta)$,  and the measure to define $\Vol(X_U)$ is different from the one in \cite{YZZ} by $2$.
Similar to  Proposition \ref{B1}, we have now
\begin{prop}\label{A1} Let $U=\prod_v U_v\subset \wh{B}^\times$ be a pure product open compact subgroup such that $\wh{\CO}^\times \subset U$. Denote by $\gamma_v=\Vol(U_0(N)_v)\Vol(U_v)^{-1}$ for all finite places $v$ and $\gamma_v=1$ for $v|\infty$. Let $\phi \in \pi_A^\JL$ be any nonzero form  and let $\alpha (W_{v}, \bar{W_{v}})$ be corresponding local constants defined in Proposition \ref{Petersson}. Let $f_1\in \pi_A, f_2\in \pi_{A^\vee}$ be any pure tensors with $(f_1, f_2)\neq 0$ and $\beta(f_{1, v}, f_{2, v})$ corresponding constants defined in \eqref{2}.  Then we have
\begin{equation}\label{final reduction GZ}
(2L_{c_1}(1, \eta) |Dc_1^2\delta|_\BA^{1/2}\nu_{c_1}^{-1})^2\cdot
\frac{\pair{P_\chi^0(f_1), P_{\chi^{-1}}^0(f_2)}_{K, L}}{(f_1, f_2)_U}=
  \frac{L'(1/2,\pi_A,\chi)}{\langle \phi,\phi \rangle_{U_0(N)}}
\prod_v \alpha_v(W_{1,v},W_{2,v})\beta_v(f_{1,v},f_{2,v})\gamma_v.
\end{equation}
\end{prop}

We will study the local factors appearing in formulas in Propositions \ref{B1} and \ref{A1} in the next section.
\subsection{Proof of Main Results}

In this subsection, we prove  Theorems \ref{GZ}, \ref{variation1},  \ref{W}, \ref{variation2} and \ref{G}, assuming local results proved in section 3.

\s{\bf  Proof of Theorem \ref{W}}. We first give a proof of the explicit Waldspurger formula.
In the equation \eqref{final reduction}, take non-zero $f_1 \in V(\pi,\chi)$, $f_2 \in V(\tilde{\pi},\chi^{-1})$,  and $\phi^0_1$ (resp. $\phi^0_2$) the normalized new vector of
$\pi^{\JL}$ (resp. $\tilde{\pi}^{\JL}$). Let $W_{\phi_i^0}:=W_i=\otimes_v W_{i, v}$ be corresponding Whittaker functions of $\phi_i^0$, $i=1, 2$. Let $R\subset B$ be the order as defined in Theorem \ref{W} and $U=\wh{R}^\times$.   Denote
\[\alpha_v := \alpha_v(W_{1,v},W_{2,v})\cdot |\delta|_v^{1/2}, \quad
\beta_v := \beta_v(f_{1,v},f_{2,v})\cdot |D\delta|_v^{-1/2}.\]
Then the equation \eqref{final reduction} becomes
\[4|Dc_1^2\delta^2|_\BA^{1/2}\nu_{c_1}^{-2}\frac{P_\chi^0(f_1)P_{\chi^{-1}}^0(f_2)}{\langle f_1,f_2\rangle_U}
  = \frac{L^{(\Sigma)}(1/2,\pi,\chi)}{\langle \phi^0_1,\phi^0_2\rangle_{U_0(N)}}
 L_{\Sigma}(1/2,\pi,\chi)L_{c_1}(1,\eta)^{-2}|c_1|_\BA^{-1}\prod_v \alpha_v\beta_v\gamma_v.\]
Let $\Sigma$ be the set in Theorem  \ref{W} and $\Sigma_\infty=\Sigma \cap \infty$ and $\Sigma_f=\Sigma\setminus \Sigma_\infty$. Comparing   with the formula \eqref{final reduction},  the proof of the explicit formula in   Theorem \ref{W} is reduced to showing that
\[L_{\Sigma_f}(1/2,\pi,\chi)L_{c_1}(1,\eta)^{-2}|c_1|_\BA^{-1}\prod_{v \nmid \infty} \alpha_v\beta_v\gamma_v
  = 2^{\# \Sigma_D}\]
and
\[L_{\Sigma_\infty}(1/2,\pi,\chi)\prod_{v|\infty} \alpha_v\beta_v\gamma_v = C_\infty^{-1},\]
which are given by Lemma \ref{p-adic} and Lemma \ref{archimedean}.

\s{\bf Proof of Theorem \ref{G}}. In the situation of Theorem \ref{G}, identify $\wt{\pi}$ with $\ov{\pi}$,  by Theorem \ref{W} we have
\[L^{(\Sigma)}(1/2,\pi,\chi) = 2^{-\# \Sigma_D + 2} (4\pi^3)^d
  \frac{\langle \phi^0, \ov{\phi^0} \rangle_{U_0(N)}}{\nu_{c_1}^2 \sqrt{|D_K|||c_1^2||}}
\frac{|P_\chi^0(f)|^2}{\pair{f_1, f_2}_U}.\]
The formula in Theorem \ref{G} follows by noting the following facts:
\begin{enumerate}
\item[(i)] $\nu_{c_1} = 2^{1-d}u_1$,
\item[(ii)] $\langle \phi^0, \ov{\phi^0} \rangle_{U_0(N)} = (2\pi)^{-d}(\phi,\phi)_{U_0(N)},$
where $\phi$ is the Hilbert new form of $\pi_A^{\JL}$. Here the last fact is obtained by applying the formula in Proposition \ref{Petersson} to $\phi$ and $\phi^0$ and the comparison of local Whittaker pairings at infinity, see remark before Proposition \ref{p-adic-beta}.
\item[(iii)] let $g_1, \cdots, g_n\in \wh{B}^\times$ be a complete set of representatives of the coset $X=B^\times \bs \wh{B}^\times/\wh{R}^\times$ and let $w_i=\# (B^\times \cap g_i \wh{R}^\times g_i^{-1}/\CO^\times)$, then as in the proof of lemma \ref{volume},  for $U=\wh{R}^\times$
    $$\qquad \qquad\pair{f, \ov{f}}_U=2^{-1}\Vol(X_U)\pair{f, \ov{f}}_\Pet=\sum_{i=1}^n |f(g_i)|^2 w_i^{-1}=\pair{\sum f(g_i)w_i^{-1} [g_i], \sum f(g_i) w_i^{-1}[g_i]}=\pair{f, f},$$
    where we identify $f$ with its image under the map $V(\pi, \chi)\lra \BC[X]$ and $\pair{\ ,\ }$ is the height pairing on $\BC[X]$.
    \end{enumerate}

\s{\bf Proof of Theorem \ref{GZ}}. To show the explicit Gross-Zagier formula in Theorem \ref{GZ},  similarly as above, we apply formula \eqref{final reduction GZ} in Proposition \ref{A1},
 to non-zero forms $f_1 \in V(\pi_A,\chi)$ and $f_2 \in V(\pi_{A^\vee},\chi^{-1})$,  $\phi^0$ the normalized new vector of $\pi_A^{\JL}$, and $U=\CR^\times$ as in Theorem \ref{GZ}.   By Lemma \ref{p-adic} and Lemma \ref{archimedean},
we have
\[L'^{(\Sigma)}(1/2,\pi,\chi) = 2^{-\# \Sigma_D + 2} (4\pi^3)^d
  \frac{\langle \phi^0, \ov{\phi^0} \rangle_{U_0(N)}}{\nu_{c_1}^2 \sqrt{|D_K|||c_1^2||}}
\frac{\pair{P_\chi^0(f_1), P_{\chi^{-1}}^0(f_2)}_{K, L}}{(f_1, f_2)_U}.\]
Then the explicit Gross-Zagier formula follows again by noting the facts (i) and (ii) above.

\s{\bf Proof of Theorem \ref{variation2} and \ref{variation1}}. We now show that the variations of Explicit Waldspurger formula in Theorem \ref{variation2} follows from the Waldspurger formula \eqref{1} and its explicit form in Theorem \ref{W}. Similarly  for variation of explicit Gross-Zagier formula in Theorem \ref{variation1}.

Let $f_1'=\otimes_v f'_{1, v}\in \pi, f_2'=\otimes_{v} f'_{2, v}\in \wt{\pi}$ be forms different from the test vectors $f_1=\otimes_v f_{1, v}\in V(\pi, \chi), f_2=\otimes_v f_{2, v} \in V(\wt{\pi}, \chi^{-1})$ from a finite set $S$ of places of $F$, respectively, such that $\pair{f'_{1, v}, f'_{2, v}}_v\neq 0$ and $\beta (f'_{1, v}, f'_{2, v})\neq 0$ for any $v\in S$. By Waldspurger formula \eqref{1}, we have the following formulas
$$\frac{P^0_\chi(f_1)\cdot P^0_{\chi^{-1}}(f_2)}{\langle f_1, f_2\rangle_U} =
\CL(\pi, \chi) \prod_v \beta(f_{1, v}, f_{2, v}), \qquad
\frac{P^0_\chi(f'_1)\cdot P^0_{\chi^{-1}}(f'_2)}{\langle f'_1, f'_2\rangle_U}=
\CL(\pi, \chi) \prod_v\beta(f'_{1, v}, f'_{2, v}),$$
where
$$\CL(\pi, \chi)=\left(\frac{\# \Pic_{K/F}(\CO_{c_1})}{2L(1, \eta)}\right)^2 \cdot \frac{2}{\Vol(X_U)}\cdot \frac{L(1/2, \pi, \chi)}{2L(1, \pi, \ad) L(2, 1_F)^{-1}}.$$

It follows that
$$\frac{P^0_\chi(f_1)\cdot P^0_{\chi^{-1}}(f_2)}{\langle f_1, f_2\rangle_U}=
\frac{P^0_\chi(f_1')\cdot P^0_{\chi^{-1}}(f'_2)}{\langle f'_1, f'_2\rangle_U}
\cdot  \prod_{v\in S}\frac{\beta(f_{1, v}, f_{2, v})}{\beta(f'_{1, v}, f'_{2, v})}.$$
The variation formula follows immediately.
\section{Local Theory}

\s{\bf Notations.} In this section,  we denote by
 $F$  a local field of characteristic zero, i.e. a finite field extension of $\BQ_v$ for some place $v$ of $\BQ$. Denote by $|\cdot|$ the absolute value of $F$ such that $d(ax)=|a| dx$ for a Haar measure $dx$ on $F$. Take an element $\delta\in F^\times$ such that $\delta \CO$ is the different of $F$ over $\BQ_v$ for $v$ finite and $\delta=1$ for $v$ infinite.   For $F$ non-archimedean,  denote by $\CO$ the ring of integers in $F$, $\varpi$ a uniformizer, $\fp$ its maximal ideal, and $q$ the cardinality of its residue field. Let $v: F\ra \BZ \cup \{\infty\}$ be the additive valuation on $F$ such that $v(\varpi)=1$. For $\mu$  a (continuous) character on $F^\times$,  denote by $n(\mu)$ the
conductor of $\mu$, that is, the minimal non-negative integer $n$ such that $\mu$ is trivial on $(1+\varpi^n\CO)\cap \CO^\times$. We will always use  the additive character $\psi$ on $F$ and the Haar measure $da$ on $F$ as in section 2 so that $da$ is self dual to $\psi$.

Denote  by $K$ a separable quadratic extension of $F$ and for any $t \in K$, write
$t \mapsto \bar{t}$ for the non-trivial automorphism of $K$ over $F$.
We use similar notations as those for $F$ with a subscript $K$. If $F$ is non-archimedean
and $K$ is non-split, denote by $e$ the ramification index of $K/F$. Denote by $\tr_{K/F}$ (resp. $\RN_{K/F}$) the trace (resp. norm) map from $K$ to $F$ and let $D\in \CO$ be an element such that  $D\CO$ is the
relative discriminant of $K$ over $F$. For an integer $c \geq 0$, denote $\CO_c$ the order $\CO + \varpi^c\CO_K$
in $K$. Let $\eta: F^\times \ra \{\pm 1\}$ be the character associated to the extension $K$ over $F$.  Let $B$ be a quaternion algebra over $F$. Let $\epsilon(B)=+1$ and $\delta(B)=0$ if $B\cong M_2(F)$ is split,  and $\epsilon (B)=-1$ and $\delta(B)=1$ if $B$ is division. Denote by $G$ the algebraic group $B^\times$ over $F$ and we also write $G$ for $G(F)$.   We take  the Haar measure  on $F^\times, K^\times$ and $K^\times/F^\times$  as in section 2. In
particular,  $\Vol(\CO^\times, d^\times a)=\Vol(\CO, da)=|\delta|^{1/2}$ and
$$\Vol(K^\times/F^\times) =
  \begin{cases}
    2, \quad &\text{if } F = \BR \text{ and } K = \BC; \\
    |\delta|^{1/2}, \quad &\text{if $K$ is the unramified extension field of $F$}; \\
    2|D\delta|^{1/2}, \quad &\text{if $K/F$ is ramified}.
\end{cases}$$
For $F$ non-archimedean and $n$ a non-negative integer, define the following subgroups of $\GL_2(\CO)$:
  \[U_0(n) := \left\{
      \matrixx{a}{b}{c}{d} \in \GL_2(\CO) \Big| c \in \fp^n\right\},
      \qquad U_1(n)=\left\{ \matrixx{a}{b}{c}{d}\in U_0(n) \Big| d \in 1+\varpi^n\CO\right\}.\]

      Let $\pi$ be an irreducible admissible representation of $G$ which is always assumed to be generic if $G\cong \GL_2$. Denote by $\omega$ the central character of $\pi$ and $\sigma = \pi^{\JL}$ the Jacquet-Langlands
correspondence of $\pi$ to $\GL_2(F)$.
Let $\chi$ be a character on $K^\times$ such that
$$\chi|_{F^\times} \cdot \omega=1.$$ For $F$ non-archimedean,  denote by
\begin{itemize}
  \item $n$ - the conductor of $\sigma$, the minimal non-negative integer such that the invariant subspace
  $\sigma^{U_1(n)}$ is nonzero.
  \item $c$ - the minimal non-negative integer $c$ such that $\chi$ is trivial on $(1+\varpi^c\CO_K)\cap \CO_K^\times$.
\end{itemize}
 Denote by
\[L(s,\pi,\chi) := L(s,\sigma \times \pi_\chi), \quad
\epsilon(s,\pi,\chi) := \epsilon(s, \sigma \times \pi_\chi,\psi)\]
the Rankin-Selberg $L$-factor and $\epsilon$-factor of $\sigma\times \pi_\chi$,  where
$\pi_\chi$ is the representation on $\GL_2(F)$ constructed from $\chi$ via
Weil representation.
Denote by $\pi_K$ the base change lifting of $\sigma$ to $\GL_2(K)$, then we have
\[L(s,\pi,\chi) = L(s,\pi_K\otimes\chi), \quad
\epsilon(s,\pi,\chi) = \eta(-1)\epsilon(s,\pi_K\otimes\chi,\psi_K)\]
Note that  $\epsilon(\pi,\chi) := \epsilon(1/2,\pi,\chi)$ equals $\pm 1$ and is independent of the choice of $\psi$.
In the following, we denote $L(s,\pi,\ad) := L(s,\sigma,\ad)$ the adjoint $L$-factor of $\sigma$.
\subsection{Local Toric Integrals}
Let  $\CP(\pi, \chi)$ denote the functional space
\[\CP(\pi,\chi) := \Hom_{K^\times}(\pi,\chi^{-1}).\]
By a theorem of Tunnell and Saito (\cite{Tunnell2}, \cite{Sat}), the space $\CP(\pi, \chi)$ has dimension at most one and equals one if and only if
\[\epsilon(\pi,\chi) = \chi\eta(-1)\epsilon(B).\]

\begin{lem} \label{Tunnell-Satio} Let the pair $(\pi,\chi)$ be as above such that
  $\epsilon(\pi,\chi) = \chi\eta(-1)\epsilon(B)$.
  \begin{enumerate}
    \item If $K$ is split or $\pi$ is a principal series, then $B$ is split.
    \item Suppose $K/F = \BC/\BR$, $\sigma$ is the discrete series of weight $k$,
      and   $\chi(z) = |z|_\BC^s(z/\sqrt{|z|_\BC})^m$ with $s \in \BC$ and
      $m \equiv k \pmod{2}$. Then $B$ is split if and only if $m \geq k$.
  \end{enumerate}
  Furthermore,   assume $F$ is nonarchimedean.
    \begin{enumerate}
      \item[(3)] If $K/F$ is nonsplit and
          $\sigma$ is the special representation $\mathrm{sp}(2)\otimes\mu$ with $\mu$ a
          character of $F^\times$, then $B$ is division if and only if $\mu_K\chi = 1$
          with $\mu_K := \mu\circ N_{K/F}$.
	\item[(4)] If $K/F$ is inert and $c=0$, then $B$ is split if and only if $n$ is even.
	\item[(5)] If $K$ is nonsplit with $c \geq n$, then $B$ is split.
    \end{enumerate}
\end{lem}
\begin{proof} See Proposition 1.6, 1.7  in \cite{Tunnell2} for (1), (3) Proposition 6.5, 6.3 (2) in \cite{Gross} for (2), (4).
  We now give a proof of (5).  If $\pi$ is a principal series, then by (1), $B$ is split.
  If $\sigma$ is a supercuspidal
 representation, then by \cite{Tunnell2} Lemma 3.1, $B$ is split if $n(\chi) \geq ne/2+(2-e)$.
 It is then easy to check that if $c \geq n$, this condition always holds.
 Finally, assume $\sigma = \mathrm{sp}(2)\otimes \mu$ with $\mu$ a character of $F^\times$. By (2), $B$ is
 division if and only if $\mu_K\chi = 1$. If $\mu$ is unramified, then $n=1$ and $\chi$ is ramified which
 implies that $B$ must be split. Assume $\mu$ is ramified, then $n = 2n(\mu)$ and by \cite{Tunnell2} Lemma 1.8,
 $fn(\mu_K) = n(\mu) + n(\mu\eta) - n(\eta)$ where $f$ is the residue degree of $K/F$. If $K/F$
 is unramified and $\mu_K\chi = 1$, then $c = n(\mu_K) = n(\mu) = n/2$, a contradiction. If
 $K/F$ is ramified and $\mu_K\chi = 1$, then $2c-1 \leq n(\mu_K) < 2n(\mu) = n$, a contradiction
 again. Hence,
 if $c \geq n$, $B$ is always split.
\end{proof}

Assume that the pair $(\pi,\chi)$ is {\em essentially unitary} in the sense that there exists some character
$\mu = |\cdot|^s$ on $F^\times$ with $s\in \BC$ such that both $\pi\otimes\mu$ and $\chi\otimes\mu_K^{-1}$
are unitary. In particular, if $\pi$ is a local component of some global cuspidal representation, then
$(\pi,\chi)$ is essentially unitary. We shall only consider essentially unitary $(\pi,\chi)$.
Under such an assumption, we  study the space $\CP(\pi,\chi)$ via the following toric integral
\[\int_{F^\times\bs K^\times} \langle \pi(t)f_1,f_2 \rangle\chi(t)dt\]
where $f_1\in \pi$, $f_2 \in \tilde{\pi}$ and $\langle \cdot,\cdot \rangle$ is any invariant pairing on
$\pi \times \tilde{\pi}$. The following basic properties for this toric integral are
established in \cite{Wal}:
\begin{itemize}
  \item It is absolutely convergent for any $f_1 \in \pi$ and $f_2 \in \tilde{\pi}$;
  \item $\CP(\pi,\chi) \not=0$ if and only if $\CP(\pi,\chi) \otimes \CP(\tilde{\pi},\chi^{-1}) \not=0$
    and in this case the above integral defines a generator of $\CP(\pi,\chi)\otimes\CP(\tilde{\pi},\chi^{-1})$.
  \item for $f_1\in \pi, f_2\in \wt{\pi}$ such that $\langle f_1, f_2 \rangle\neq 0$,  define the toric integral
  \[\beta(f_1,f_2) := \frac{L(1,\eta)L(1,\pi,\ad)}{L(2,1_F)L(1/2,\pi,\chi)}
   \int_{F^\times\bs K^\times} \frac{\langle \pi(t)f_1,f_2 \rangle}{\langle f_1,f_2 \rangle}\chi(t)dt.\]
  Then $\beta(f_1, f_2)=1$ in the case:  $B=M_2(F)$, $K$ is an unramified extension over $F$, both $\pi$ and $\chi$ are unramified, $dt$ is normalized such that $\Vol(\CO_K^\times/\CO^\times)=1$, and $f_1, f_2$ are spherical.
\end{itemize}

Note that for any pair $(\pi,\chi)$, the above $\beta$ integral is invariant if we modify $(\pi,\chi)$
to $(\pi\otimes \mu, \chi\otimes \mu_K^{-1})$ for any character $\mu$ of $F^\times$. Therefore, we may
assume $\pi$ and $\chi$ are both unitary from now on and identify $(\tilde{\pi},\chi^{-1})$
with $(\bar{\pi},\bar{\chi})$.  Let $(\ ,\ ): \pi \times \pi \lra \BC$ be the Hermitian pairing  defined by $(f_1, f_2)=\pair{f_1, \ov{f_2}}$.

Denote $\beta(f) := \beta(f,\bar{f})$. Then
the functional space $\CP(\pi,\chi)$ is nontrivial if and only if $\beta$ is nontrivial.
Assume $\CP(\pi,\chi)$ is nonzero in the following.
A nonzero vector $f$ of $\pi$ is called a {\em test vector} for $\CP(\pi,\chi)$ if
$\ell(f) \not=0$ for some (thus any) nonzero $\ell \in \CP(\pi,\chi)$, or equivalently,
$\beta(f)$ is non-vanishing.

The notation of new vectors in an irreducible smooth admissible representation of $\GL_2(F)$
(see  \cite{Cass} for $F$ non-archimedean and \cite{Popa} for $F$ archimedean) can be viewed as a special case of test vectors. Let $\pi$ be an irreducible admissible representation of $\GL_2(F)$.
Recall the definition of {\em new vector line} in  $\pi$  as follows. Denote by $T = K^\times$
the diagonal torus in $\GL_2(F)$. Write $T = Z T_1$ with $T_1 = \left\{
    \begin{pmatrix}
      * & \\
        & 1
    \end{pmatrix}
  \right\}$.
  \begin{itemize}
    \item If $F$ is nonarchimedean, then the new vector line is the invariant subspace
     $\pi^{U_1(n)}$.
    \item If $F$ is archimedean, take $U = O_2(\BR)$ if $F = \BR$ and $U_2$ if
       $F = \BC$. The new vector line consists of vectors $f \in \pi$ which are
       invariant under $T_1\cap U$ with weight minimal.
  \end{itemize}
 It is known that new vectors satisfy the following properties.
  \begin{enumerate}
  \item For any $s \in \BC$, denote by $\omega_s$ the character on $T$ such that
  $\omega_s|_Z = \omega$ and $\omega_s|_{T_1} = |\cdot|^{s-1/2}$.
  Then any nonzero $f$ in the new vector line is a test vector for
  $\CP(\pi,\omega_s^{-1})$;
  \item If denote by $\CW(\pi,\psi)$ the Whittaker model of $\pi$ with respect to
  $\psi$, then there is  a vector $W_0$ in
  the new vector line, called the {\em normalized new vector} of $\pi$ such that the local
  zeta integral $|\delta|^{s-1/2}Z(s,W_0)$ equals $L(s,\pi)$.
  \end{enumerate}

\subsection{Local Orders of Quaternions}

Assume $F$ is nonarchimedean in this subsection.

First, in the case
that the quaternion algebra $B$ is split, given non-negative integers $m$ and $k$,
we want to classify all the $K^\times$
conjugacy classes of Eichler orders $R$ in $B$ with discriminant $m$ such that
$R\cap K = \CO_k$. For this,  identify $B$ with the $F$-algebra $\End_F(K)$ which
contains $K$ as an $F$-subalgebra by multiplication.
Recall that an Eichler order in $B$ is the intersection of two maximal orders in $B$.
Then any Eichler order must be of form $R(L_1, L_2):=R(L_1)\cap R(L_2)$ where $L_i$,
$i=1,2$ are two $\CO$-lattices in $K$ and $R(L_i) := \End_\CO(L_i)$.
Denote by $d(L_1, L_2)$ its discriminant.
For any maximal order $R(L)$, there exists a unique integer $j\geq 0$ such that $L=t\CO_j$
for some $t\in K^\times$. In fact, $\CO_j=\{x\in K | xL \subset L\}$. Thus any $K^\times$-conjugacy class
of Eichler order contains an order of form $R(\CO_j, t\CO_{j'})$ with $0 \leq j'\leq j$ and
$t\in K^\times$ and the conjugacy class is exactly determined by the integers $j'\leq j$ and
the class of $t\in K^\times$ modulo $F^\times \CO_{j'}^\times$.
The question is reduced to solving the equation with variables $k'$ and $[t]$:
$$d(\CO_k, t\CO_{k'})=m, \qquad 0\leq k'\leq k,\quad [t]\in K^\times/F^\times \CO_{k'}^\times.$$
Note that if $(k', [t])$ is a solution, then so is $(k', [\ov{t}])$.
A complete representative system $(k', t)$ with $t\in K^\times$ for solutions to the above equation
corresponds to a complete system  $R(\CO_k, t\CO_{k'})$ for $K^\times$-conjugacy classes of
Eichler orders $R$ with discriminant $m$ and $R\cap K=\CO_k$.

\begin{lem}\label{p}Let $m, k$ be non-negative integers.  Let $\tau\in K^\times$ such that $\CO_K=\CO[\tau]$, if $K$ split then $\tau^2-\tau=0$,  and if $K$ non-split then $v(\tau)=(e-1)/2$. Denote by $d:= k + k' - m$.
  Then a complete representative system of $(k', t)$ is the following:
\begin{itemize}
\item For $0\leq m\leq 2k$,    $k'\in [|m-k|, k]$ with $d$ even, i.e. $d\in 2\cdot [0, k']$,  and
$$t=1+\varpi^{\frac{d}{2}}\tau u, \qquad u\in (\CO/\varpi^{k'-\frac{d}{2}}\CO)^\times.$$
Note that in the case $k'=k-m\geq 0$, the unique class of $t$ is also represented by $1$.
\item For split $K\cong F^2$ and $k+1\leq m$,   $k'\in [0, \min(m-k-1, k)]$, i.e. $d\in [k-m, 0)$,
   and $$t=(\varpi^{\pm d}u, 1),  \qquad u\in (\CO/\varpi^{k'}\CO)^\times.$$
\item For non-split $K$ and $k+1\leq m\leq 2k+e-1$,  $k'=m-k-e+1$, i.e. $d=1-e$,  and
$$t=\varpi x+\tau, \qquad x\in \CO/\varpi^{k'+e-2}\CO.$$
\end{itemize}
\end{lem}
\begin{proof} Note that the discriminant $d(L_1, L_2)$ of the Eichler order $R(L_1, L_2)$  can be computed as follows. Let $e_i, e_i'$ be an $\CO$-basis of $L_i$, $i=1, 2$,  and let $A=(a_{ij})\in \GL_2(F)$ such that $A\begin{pmatrix} e_1\\ e_1'\end{pmatrix}=\begin{pmatrix} e_2\\ e_2'\end{pmatrix}$. Let $v: F\ra \BZ\cup \{\infty\}$ be the additive valuation on $F$ such that $v(\varpi)=1$. Denote by $\alpha=\min_{i, j} v(a_{ij})$ and $\beta=v(\det A)$.  Then  $d(L_1, L_2)=|2\alpha-\beta|$. Now solve the equation
$$d(\CO_k, t\CO_{k'})=m, \qquad  k'\in [0, k], \quad t\in K^\times/F^\times\CO_{k'}^\times.$$
\end{proof}

Denote
\[c_1 =
  \begin{cases}
    0, \quad &\text{if  $K$ is nonsplit and $c < n$}; \\
    c, \quad &\text{otherwise}.
\end{cases}\]

\begin{lem}\label{lem-order}
  There exists an order $R$ of discriminant $n$ and $R\cap K = \CO_{c_1}$
  satisfying the condition: if $nc_1 \not=0$, then $R$ is the intersection of two maximal
  orders $R'$ and $R''$ of $B$ such that $R' \cap K = \CO_{c_1}$,
  $R'' \cap K = \CO_{\max\{0,c_1-n\}}$. Such order is unique up to
  $K^\times$-conjugacy unless $0 < c_1 < n$. In the case $0 < c_1 < n$,
  there are exact two $K^\times$-conjugacy classes which are conjugate to
  each other by a normalizer of $K^\times$.
\end{lem}
\begin{proof}
If $nc_1 = 0$, this is proved in \cite{Gross}, Propositions 3.2 and
3.4. Now assume that $nc_1\neq 0$, then $B$ is split and one can
apply Lemma \ref{p}.
\end{proof}

Let $R$ be an $\CO$-order of $B$ of discriminant $n$ such that
$R \cap K = \CO_{c_1}$. Such an order $R$ is called admissible for
$(\pi,\chi)$ if the following conditions are satisfied
\begin{enumerate}
  \item If $nc_1 \not=0$ (thus $B$ is split), then
  $R$ is the intersection of two maximal orders $R'$ and $R''$ of $B$ such that
  $R'\cap K = \CO_{c_1}$ and $R''\cap K = \CO_{\max\{0,c_1-n\}}$.
\item If $0 <c_1 < n$, fix an $F$-algebra isomorphism $K \cong F^2$ and
  identify $B$ with $\End_F(K)$. Note that the two $K^\times$-conjugacy
  classes of $\CO$-orders in $B$ satisfying the above condition (1) contain respectively
  the orders $R_i = R'_i\cap R''_i, i=1, 2$ with $R'_1 =R_2'= \End_\CO(\CO_c),$
  $R''_1 = \End_\CO( (\varpi^{n-c},1)\CO_K)$ and $R''_2 = \End_\CO( (1,\varpi^{n-c})\CO_K)$.
  Denote by $\chi_1(a) = \chi(a,1)$ and $\chi_2(b) = \chi(1,b)$. Then $R$ is $K^\times$-conjugate
  to some $R_i$  such that the conductor of $\chi_i$ is $c_1$.
\end{enumerate}

\begin{lem}\label{order-uniqueness}
  If $K$ is nonsplit, $n>0$ and $c=0$, then there is a unique admissible order $R$ for
  $(\pi,\chi)$.
\end{lem}
\begin{proof}
   Let $\CO_B$ be a maximal order containing $\CO_K$, then by \cite{Gross} (3.3), any
   admissible order for $(\pi,\chi)$ is $K^\times$-conjugate to
   $R := \CO_K + I\CO_B$ where $I$ is a nonzero ideal of $\CO_K$ such that
   $n = \delta(B) + \length_\CO(\CO_K/I)$. If $B$ is nonsplit, then $\CO_B$ is
   invariant under $B^\times$-conjugations and $R$ is unique.
   Assume $B$ is split. As $\CO_K^\times \subset \CO_B^\times$, $\CO_B$ is
   invariant under $F^\times\CO_K^\times$-conjugations.
   In particular, if $K$ is unramified, then $K^\times = F^\times\CO_K^\times$
   and $R$ is unique. Consider the case $K$ is ramified. Then
   $K^\times = F^\times\CO_K^\times \cup \varpi_KF^\times\CO_K^\times$ and
   it reduces to show that $\varpi_K$ normalizes $R$. For this, embed $K$ into
   $B = M_2(F)$ by $\varpi_K \mapsto
   \begin{pmatrix}
     \tr \varpi_K & 1 \\
     -N\varpi_K & 0
   \end{pmatrix}$ and take $\CO_B = M_2(\CO)$. Then $R = \CO_K + \varpi_K^nM_2(\CO)$.
   Note that $R_0(1) = \CO_K + \varpi_KM_2(\CO)$ with $R_0(1) =
   \begin{pmatrix}
     \CO & \CO \\
     \fp & \CO
   \end{pmatrix}$ the Iwahori order in $M_2(F)$. Denote $m$ the maximal integer such that
   $2m \leq n$. Then $R = \CO_K +\varpi^{m-1}\varpi_K R_0(1)$ (resp. $R = \CO_K + \varpi^mR_0(1)$)
   if $n$ is even (resp. $n$ is odd). As $\varpi_K$ normalizes $R_0(1)$, it also
   normalizes $R$ and $R$ is unique.
\end{proof}

In the following,  take an admissible $\CO$-order $R$ of $B$.
Let $U = R^\times$ and define
  \[\gamma := \frac{\Vol(U)}{\Vol(U_0(n))},\]
  where the Haar measure is given so that $\Vol(\GL_2(\CO)) = L(2,1_F)^{-1}|\delta|^2$
  and $\Vol(\CO_B^\times) = L(2,1_F)^{-1}(q-1)^{-1}|\delta|^2$ if $B$ is division.

\begin{lem}\label{gamma}
   If either $R$ is not maximal or $B$ is nonsplit, then
  \[\gamma = L(1,1_F)(1-e(R)q^{-1})\]
  where $e(R)$ is the Eichler symbol of $R$, which is defined as follows. Let
    $\kappa(R) = R/\mathrm{rad}(R)$ with $\mathrm{rad}(R)$ the Jacobson radical of $R$ and let
    $\kappa$ be the residue field of $F$. Then
    \[e(R) =
      \begin{cases}
	1, \quad &\text{if $\kappa(R) = \kappa^2$,} \\
	-1, \quad &\text{if $\kappa(R)$ is a quadratic field extension of $\kappa$}, \\
	0, \quad &\text{if $\kappa(R) = \kappa$}.
    \end{cases}\]
\end{lem}
\begin{proof}
 Let $R_0$ be a maximal order of $B$ containing $R$. Then we have the following  formula (for example, see \cite{Yu}):
   $$\frac{[R_0^{\times}: R^\times]}{[R_0: R]} = \frac{|\kappa(R_0)^\times|/|\kappa(R^\times)|}{|\kappa(R_0)|/\kappa(R)|}.$$
  If $B$ is split and $R$ is not maximal, then
  \[ [R_0:R] = q^n, \quad \frac{|\kappa(R_0)^\times|}{|\kappa(R_0)|}
   = (1-q^{-2})(1-q^{-1}),
  \quad \frac{|\kappa(R)|}{|\kappa(R)^\times|} = (1-q^{-1})^{-1}(1-e(R)q)^{-1},\]
  while if $B$ is division, then
  \[ [R_0:R] =  q^{n-1}, \quad \frac{|\kappa(R_0)^\times|}{|\kappa(R_0)|}
   = 1-q^{-2},
  \quad \frac{|\kappa(R)|}{|\kappa(R)^\times|} = (1-q^{-1})^{-1}(1-e(R)q)^{-1}.\]
  Summing up,
  \[ [R_0^{\times}: R^\times] =
  (q-1)^{-\delta(B)} q^n (1-q^{-2})(1-e(R)q^{-1})^{-1},\]
    where $\delta(B) = 0$ (resp. $1$) if $B$ is split (resp. ramified). Thus
    \[
    \begin{aligned}
      \gamma^{-1} &= \frac{\Vol(U_0(n))}{\Vol(U)} = \frac{\Vol(\GL_2(\CO))}
      {\Vol(R_0^\times)}\frac{[R_0^\times:U]}{[\GL_2(\CO):U_0(n)]} \\
    &= \frac{L(2,1)^{-1}}{(q-1)^{-\delta(B)}L(2,1)^{-1}}
    \frac{(q-1)^{-\delta(B)}q^n(1-q^{-2})(1-e(R)q^{-1})^{-1}}{q^n(1-q^{-2})(1-q^{-1})^{-1}} \\
  &= L(1,1_F)^{-1}(1-e(R)q^{-1})^{-1}.
  \end{aligned}
  \]
\end{proof}

\subsection{Test Vector Spaces}
\begin{defn}
Define  $V(\pi,\chi)\subset \pi$ to be the subspace of vectors $f$ satisfying the following condition:
\begin{itemize}
  \item for nonarchimedean $F$,  $K$ split or $c \geq n$,  let $U\subset G$ be the compact subgroup
    defined before Lemma \ref{gamma}, then $f$ is  $\omega$-eigen under $U$. Here,  write $U = (U\cap Z) U'$
       such that $U' = U$ if $cn=0$ and  $U'\cong U_1(n)$ otherwise, and view $\omega$ as a
       character on $U\cap Z$ and extends to $U$ by making it trivial on $U'$;
    \item for nonarchimedean $F$, $K$ nonsplit and $c < n$, $f$ is
    $\chi^{-1}$-eigen under the action of $K^\times$;
   \item for archimedean $F$,  let $U$ be a maximal compact subgroup of $G$ such that $U\cap K^\times$
     is the maximal compact subgroup of $K^\times$, then $f$ is
     $\chi^{-1}$-eigen under $U\cap K^\times$ with weight minimal.
\end{itemize}
\end{defn}

\begin{prop}\label{local-multi-one}
  The dimension of $V(\pi,\chi)$ is one and
  any nonzero vector in $V(\pi,\chi)$ is a test vector for $\CP(\pi,\chi)$.
\end{prop}
\begin{proof}
  If $F$ is nonarchimedean, the claim that $\dim V(\pi,\chi) =1$ follows from local new form theory
  \cite{Cass}. Assume $F$ is archimedean. If $K$ is nonsplit, then $V(\pi,\chi)$ is the $\chi^{-1}$-eigen
  line. If $K$ is split, without loss of generality, embed $K^\times$ into $G\cong \GL_2(F)$ as the
  diagonal matrices and decompose $K^\times = F^\times K^1$ such that the image of $K^1$ in $G$ is
  $\left\{
    \begin{pmatrix}
      * & \\
        & 1
    \end{pmatrix}
  \right\}$. Then $V(\pi,\chi)$ is the new vector line for $\pi\otimes\chi_1$ with
  $\chi_1 := \chi|_{K^1}$.

  We  shall prove any nonzero vector in
  $V(\pi,\chi)$ is a test vector in next subsection by computing the toric integral $\beta$.
\end{proof}

\begin{prop}\label{local-multi-one-2}
    Assume $K/F$ is a quadratic extension of nonarchimedean fields with $n > 0$ and $c=0$.
    Then $V(\pi,\chi) \subseteq \pi^{R^\times}$ and $\dim \pi^{R^\times}
     = \dim \pi^{\CO_K^\times} \leq 2$. The dimension of $\pi^{R^\times}$ is one precisely when
     $K/F$ is inert or $K/F$ is ramified and $\epsilon(\pi,\chi_1) \not=
     \epsilon(\pi,\chi_2)$ where $\chi_i, i=1, 2,$ are unramified characters of $K^\times$ with
     $\chi_i|_{F^\times}\cdot \omega = 1$.
\end{prop}

  The proof of this proposition is refered to \cite{Gross} and \cite{GP}  except for the case that
  $\pi$ is a supercuspidal representation on $G = \GL_2(F)$.  For this case, the proof in \cite{Gross}, \S 7
  is based on a character formula for odd residue characteristic.
  We next prove this case with arbitrary residue characteristic.

  Let  $R_0 = M_2(\CO)$ if $e=1$ and  the Iwahori order $
  \begin{pmatrix}
    \CO & \CO \\
    \fp & \CO
  \end{pmatrix}$ if $e=2$. Fix an embedding of $K$ into $M_2(F)$ such that
  $R_0 \cap K = \CO_K$. Consider the following filtration of open compact subgroups of
  $G$ and $K^\times$:
  \[\CK(r) := (1+\varpi^r R_0) \cap \GL_2(\CO), \quad \CE(r) := \CK(r)\cap K^\times, \quad r \geq 0.\]
  Denote $m$ the minimal integer such that $2m + 1 \geq n$. The proof is based on the following proposition.

  \begin{prop}\label{prasad-lemma}
  For any integer $r \geq m$, $\pi^{\CK(r)} = \pi^{\CE(r)}.$
\end{prop}
\begin{proof}
  Firstly, note that it is enough to prove Proposition \ref{prasad-lemma}
  for the case $\pi$ is minimal, that is, $\pi$ has minimal conductor among its twists.
  In fact, assume $\pi$ is not minimal. Denote $n_0$ the minimal conductor of $\pi$.
  Take a character $\mu$ so that $\pi_0 := \pi \otimes \mu$ has conductor $n_0$.
  Then by \cite{Tunnell1}, Proposition 3.4,
  $n_0 \leq \max(n,2n(\mu))$ with equality if
  $\pi$ is minimal or $n \not= 2n(\mu)$.
  In particular, $n = 2m$ with $n(\mu) = m$. Hence, for any $r \geq m$,
  $\pi^{\CK(r)} = \pi_0^{\CK(r)}, \quad \pi^{\CE(r)} = \pi_0^{\CE(r)}$.
  Note that $r \geq n_0/2$ and one then can apply the proposition for the minimal representation $\pi_0$.

  Assume $\pi$ is minimal in the following. As $\CK(r) \supset \CE(r)$,  $\pi^{\CK(r)} \subset \pi^{\CE(r)}$.
  It remains to prove that $\pi^{\CK(r)}$ and $\pi^{\CE(r)}$ have the same dimension.  Denote
 $\pi_D$ the representation on $D^\times$ where $D$ is the division quoternion algebra over $F$
so that the Jacquet-Langlands lifting of $\pi_D$ to $G$ is $\pi$. Then $\pi_D$ has conductor $n$, that is,
$\pi_D^{1+\varpi_D^{n-1}\CO_D} = \pi_D$ and $\pi_D^{1+\varpi_D^{n-2}\CO_D} = 0$ where $\varpi_D$ is
a uniformizer of $D$. Moreover, by \cite{Ca}, Proposition 6.5,
\[\dim \pi_D =
  \begin{cases}
    2q^{m-1}; \quad &\text{if $n$ is even}; \\
    q^{m} + q^{m -1}; \quad &\text{if $n$ is odd}.
  \end{cases}
\]
Note that for any $r \geq m$, $\CE(r) \subset (1+\varpi_D^{n-1}\CO_D) \cap \CO_K^\times$.
Therefore, by Tunnell-Saito's theorem, if we denote $\CX(r)$ the set
of all the characters $\mu$ on $K^\times$ such that
$\mu|_{F^\times}\omega =1$ and $\mu|_{\CE(r)} = 1$, then
\[
  \begin{aligned}
 \dim \pi^{\CE(r)} + \dim \pi_D
  &= \sum_{\mu \in \CX(r)} \dim \pi^\mu + \sum_{\mu} \dim \pi_D^\mu
  = \sum_{\mu \in \CX(r)} (\dim \pi^\mu + \dim \pi_D^\mu) = \# \CX(r)
  \end{aligned}
\]
and on the other hand, the lemma below implies that
\[\dim \pi^{\CK(r)} + \dim \pi_D = \# \CX(r),\]
and then the equality $\dim \pi^{\CE(r)}=\dim \pi^{\CK(r)}$ holds.
\end{proof}
\begin{lem}
  Let $\pi$ be minimal.
  For any integer $r \geq m$, we have the following dimension formula
  \[\dim \pi^{\CK(r)} =
    \begin{cases}
      q^r + q^{r-1} -2q^{m-1}; \quad &\text{if $n$ is even and $e=1$}; \\
      q^r + q^{r-1} - (q^{m-1} + q^{m-2}); \quad &\text{if $n$ is odd and $e=1$}; \\
      2q^r - (q^m + q^{m-1}); \quad &\text{if $n$ is odd and $e=2$}; \\
      2q^r - 2q^{m-1}; \quad &\text{if $n$ is even and $e=2$}. \\
    \end{cases}\]
  \end{lem}
\begin{proof}
For $r = m$ and $e=1$, this formula occurs in \cite{Cass2}, Theorem 3.
We now use the method in \cite{Cass2} to prove the dimension formula
for the case $n$ is even and $e = 1$ while other cases are similar.
Firstly, recall some basics about Kirillov model.
Let $\psi$ be an unramified additive character of $F$. Associated to $\psi$,
we can realize $\pi$ on the space $C_c^\infty(F^\times)$ of Schwartz functions on the multiplicative group.
For any
$f \in C_c^\infty(F^\times)$ and any character $\mu$ of $\CO^\times$, define
\[f_k(\mu) = \int_{\CO^\times} f(u\varpi^k)\mu(u)du\]
where we  choose the Haar measure on $\CO^\times$ such that the total measure is $1$. Define further the formal
power series
\[\hat{f}(\mu,t) = \sum_{k \in \BZ} f_k(\mu)t^k\]
which is actually a Laurent polynomial in $t$ as $f$ has compact support on $F^\times$. Because $f$ is
locally constant,
this vanishes identically for all but a finite number of $\mu$. By Fourier duality for $F^\times$,
knowing $f(\mu,t)$
for all $\mu$ is equivalent to knowing $f$. For each $\mu$, there is a formal power series $C(\mu,t)$
such that for all
$f \in C_c^\infty(F^\times)$
\[(\pi(w)f)^{\hat{}}(\mu,t) = C(\mu,t)\hat{f}(\mu^{-1}\omega_0^{-1},t^{-1}z_0^{-1}),
  \quad C(\mu,t) = C_0(\mu)t^{n_\mu},
\quad w =
\begin{pmatrix}
  0 & 1\\
  -1 & 0
\end{pmatrix}
\]
where $\omega_0 = \omega|_{\CO^\times}$, $z_0 = \omega(\varpi)$ and some integer $n_\mu \leq -2$. Moreover,
if $\mu = 1$, then $-n_1 = n$. For any character $\mu$ of $\CO^\times$,
  \[-n_\mu =
    \begin{cases}
      n, \quad &\text{if } n(\mu) \leq m; \\
      2n(\mu), \quad &\text{if } n(\mu) > m.
  \end{cases}\]
  In fact, if take any character $\Omega$ on $F^\times$ so that $\Omega|_{\CO^\times} = \mu$,
  denote $\pi' = \pi\otimes\Omega$ and $C'(\cdot,\cdot)$ the monomial occured in the above functional equation,
  then for any character $\nu$ on $\CO^\times$, $C'(\nu,t) = C(\nu\mu, \Omega(\varpi)t)$. Therefore,
  $-n_\mu = n(\pi') = \max(n,2n(\mu))$.

  On the other hand, by \cite{Cass2}, Corollary to Lemma 2, for any
   $r \geq m$, the subspace $\pi^{\CK(r)}$ is isomorphic to the space of all functions $f(\mu,t)$ such that
  \begin{enumerate}
    \item $\hat{f}(\mu,t) = 0$ unless $n(\mu) \leq r$;
    \item for each $\mu$, $f_k(\mu) = 0$ unless $-r \leq k \leq n_\mu + r$.
  \end{enumerate}

Summing up, for a given $\mu$ with conductor $n(\mu) \leq r$, the dimension of the space consisting of
$\hat{f}(\mu,t)$ with $f \in \pi^{\CK(r)}$ is
\[\begin{cases}
    2(r-m) + 1; \quad &\text{if } n(\mu) \leq m; \\
    2(r-n(\mu)) + 1; \quad &\text{if } n(\mu) > m.
\end{cases}\]
Therefore,
\[
  \begin{aligned}
\dim \pi^{\CK(r)} &= (q^m - q^{m-1})(2(r-m) + 1) + \sum_{m < k \leq r} (q^k - 2q^{k-1} + q^{k-2})(2(r-k) + 1) \\
  &= q^r + q^{r-1} - 2q^{m-1}.
\end{aligned}
\]
\end{proof}

\begin{proof}[Proof of Proposition \ref{local-multi-one-2}]
   Note that $R^\times = \CO_K^\times\CK(m)$ unless $K$ is ramified with $n$ even and
   once this equation holds, Proposition \ref{local-multi-one-2} follows directly from
   Proposition \ref{prasad-lemma}. It remains to consider the case $K$ is ramified with $n$ even.
   For this case, $R^\times = \CO_K^\times\CK'(m)$ with $\CK'(m) = 1+\varpi_K^{2m-1}R_0$.
   We want to show $\pi^{\CK'(m)} = \pi^{\CE'(m)}$ with $\CE'(m) = \CK'(m)\cap K^\times$ and
   Proposition \ref{local-multi-one-2} then holds.
   By \cite{Tunnell2}, Proposition 3.5, $\pi$ is not minimal. Take a character $\mu$ so that
   $\pi_0 = \pi\otimes\mu$ has minimal conductor $n_0$. Then $n(\mu) = m$. Apply Proposition
   \ref{prasad-lemma},
   \[\pi^{\CK'(m)} = \pi_0^{\CK'(m)} \supset \pi_0^{\CK(m-1)} = \pi_0^{\CE(m-1)}.\]
   We claim that $\pi_0^{\CE(m-1)} = \pi_0^{\CE'(m)}$. If so,
   $\pi_0^{\CE(m-1)} = \pi^{\CE'(m)}$ and then $\pi^{\CK'(m)} = \pi^{\CE'(m)}$. To prove this,
   note that $\CE'(m) \subset \CE(m-1) \subset 1+\varpi_D^{n_0-1}\CO_D$. Use Tunnell-Saito's theorem,
   \[\dim \pi_0^{\CE(m-1)} + \dim \pi_{0,D} = \# \CX(m-1), \quad
   \dim \pi_0^{\CE'(m)} + \dim \pi_{0,D} = \# \CX'(m)\]
   where the set $\CX(m-1)$ consists of characters $\Omega$ of $K^\times$ such that
   $\Omega|_{F^\times}\cdot \omega_{\pi_0} = 1$ with $\Omega|_{\CE(m-1)} = 1$ and the
   set $\CX'(m)$ is defined similarly. As they are nonempty,
   \[\# \CX(m-1) = \# K^\times/F^\times\CE(m-1) =  \# K^\times/F^\times\CE'(m) = \# \CX'(m).\]
   Thus $\pi_0^{\CE(m-1)} = \pi_0^{\CE'(m)}$ and the proof is complete.
\end{proof}

\subsection{Local Computations}

Let $\CW(\sigma,\psi)$ be the Whittaker model of $\sigma$ with respect to $\psi$ and
recall that we have an  invariant Hermitian form on $\CW(\sigma, \psi)$ defined by
\[(W_1, W_2):= \int_{F^\times} W_1\left[
  \begin{pmatrix}
    a & \\
      & 1
  \end{pmatrix}
  \right]
  \ov{W_2\left[
    \begin{pmatrix}
      a & \\
        & 1
    \end{pmatrix}
  \right]}d^\times a.\]
For any $W \in \sigma$, denote
\[\alpha(W) = \frac{( W,W )}{L(1,\sigma,\ad)L(1,1_F)L(2,1_F)^{-1}}.\]

\begin{prop}\label{local inner product}
  Denote $W_0$ the normalized new vector of $\sigma$.
  If $F$ is nonarchimedean, then
  \[\alpha(W_0)|\delta|^{1/2} =
    \begin{cases}
      1, \quad &\text{if $\sigma$ is unramified}; \\
      L(2,1_F)L(1,1_F)^{-1}L(1,\sigma,\ad)^{-\delta_\sigma}, \quad &\text{otherwise},
  \end{cases}\]
  where $\delta_\sigma \in \left\{ 0,1 \right\}$ and equals $0$ precisely when $\sigma$ is a subrepresentation
  of induced representation $\Ind(\mu_1,\mu_2)$ with at least one $\mu_i$ unramified. If $F = \BR$
  and $\sigma$ is the discrete series $\CD_\mu(k)$, then $\alpha(W_0) = 2^{-k}.$
\end{prop}

The above proposition follows from the  explicit form of $W_0$.
If $F$ is nonarchimedean,  $W_0$ is the one in the new vector line such that
\[W_0\left[
  \begin{pmatrix}
    \delta^{-1} & \\
      & 1
  \end{pmatrix}
\right] = |\delta|^{-1/2}\]
and we have the following list (See \cite{Sch}, p.23)
\begin{enumerate}
  \item If $\sigma = \pi(\mu_1,\mu_2)$ is a principal series, then
    \[W_0\left[
      \begin{pmatrix}
	y & \\
	  & 1
      \end{pmatrix}
    \right] =
    \begin{cases}
      |y|^{1/2}\sum_{{k+l= v(y\delta)}\atop {k,l \geq 0}}
      \mu_1(\varpi)^k\mu_2(\varpi)^l 1_\CO(\delta y), \quad \text{if $n(\mu_1) = n(\mu_2) = 0$}; \\
      |y|^{1/2}\mu_1(\delta y)1_\CO(\delta y), \quad \text{if $n(\mu_1) = 0$ and $n(\mu_2) > 0$}; \\
      |\delta|^{-1/2}1_{\CO^\times}(\delta y), \quad \text{if $n(\mu_1)>0$ and $n(\mu_2)>0$}.
  \end{cases}\]
   \item If $\sigma = \mathrm{sp}(2)\otimes \mu$ is a special representation, then
     \[W_0\left[
      \begin{pmatrix}
	y & \\
	  & 1
      \end{pmatrix}
    \right] =
    \begin{cases}
       |\delta|^{-1/2}\mu(\delta y)|\delta y|1_\CO(\delta y), \quad &\text{if $n(\mu) = 0$}; \\
       |\delta|^{-1/2}1_{\CO^\times}(\delta y), \quad &\text{if $n(\mu)>0$}.
  \end{cases}\]
  \item If $\sigma$ is supercuspidal, then
    \[ W_0\left[
      \begin{pmatrix}
	y & \\
	  & 1
      \end{pmatrix}
    \right] = |\delta|^{-1/2}1_{\CO^\times}(\delta y).\]
\end{enumerate}
If $F = \BR$ and $\sigma$ is the discrete series $\CD_\mu(k)$, then
\[W_0\left[
  \begin{pmatrix}
    y & \\
      & 1
  \end{pmatrix}
\right] = |y|^{k/2}e^{-2\pi |y|},\]
and in general, for archimedean cases, it is expressed by the Bessel function \cite{Popa}.
Note that for $F=\BR$ and $\sigma$ a unitary discrete series of weight $k$,
let $W\in \CW(\sigma, \psi)$ be the vector satisfying
\[
W\left[
  \begin{pmatrix}
    y & \\
      & 1
  \end{pmatrix}
\right] = |y|^{k/2}e^{-2\pi |y|}1_{\BR_+^\times}(y).\]
Then $W$ can be realized as a local component of a Hilbert newform and
\[(W_0,W_0) = 2 (W,W), \quad Z(s,W) = \frac{1}{2}L(s,\sigma).\]

\begin{prop}\label{p-adic-beta}
  If $F$ is nonarchimedean, let $f$ be a nonzero vector in the one-dimensional space
  $V(\pi,\chi)$, then
  \[
   \beta(f)|D\delta|^{-1/2} =
   \begin{cases}
     1,  \quad \text{if $n = c = 0$}; \\
     L(1,\eta)^2|\varpi^c|,  \quad \text{if $n=0$ and $c>0$}; \\
     \displaystyle{\frac{L(1,1_F)}{L(2,1_F)}L(1,\pi,\ad)^{\delta_\pi}}, \quad \text{if $n>0$, $c=0$ and $K$ is split};\\
     \displaystyle{\frac{L(1,1_F)}{L(2,1_F)}L(1,\eta)^2|\varpi^c|
     \frac{L(1,\pi,\ad)^{\delta_\pi}}{L(1/2,\pi,\chi)}},
     \quad \text{if $nc>0$, either $K$ is split or  $c \geq n$};\\
     \displaystyle{e(1-q^{-e})\frac{L(1,\pi,\ad)}{L(1/2,\pi,\chi)}}, \quad \text{if $n>c$ and $K$ is nonsplit}.
 \end{cases}\]
 which is independent of the choice of $f \in V(\pi,\chi)$.
\end{prop}

The proof of Proposition \ref{p-adic-beta}  is reduced to computing the integral
\[\beta^0 = \int_{F^\times \bs K^\times} \frac{( \pi(t)f,f )}{( f,f )}
\chi(t)dt,\]
where $f$ is any nonzero vector in $V(\pi,\chi)$.

In the case  $n>c$ and $K$ is nonsplit,
$f$ is $\chi^{-1}$-eigen and it is easy to see that $\beta^0 = \Vol(F^\times \bs K^\times)$.

From now on assume $n \leq c$ or $K$ is split. Then $B  = M_2(F)$ by Lemma \ref{Tunnell-Satio} (5).
Recall that the space $V(\pi, \chi)$ depends on a choice of an admissible order $R$ for $(\pi, \chi)$.  Let $f$ be a test vector in $V(\pi, \chi)$ defined by $R$. For any $t\in K^\times$, $f':=\pi(t) f$ is a test vector defined by
the admissible order $R' = tRt^{-1}$. It is easy to check that $\beta(f') = \beta(f)$. Thus, for a $K^\times$-conjugacy class of admissible orders, we can pick
a particular order  to compute $\beta^0$.
Note that there is a unique  $K^\times$-conjugacy class of admissible orders unless in the exceptional case
$0 < c_1 < n$ and $n(\chi_1) = n(\chi_2) = c$. In the exceptional case, there
are exactly two $K^\times$-conjugacy classes of admissible orders, which are conjugate to each other by a normalizer of $K^\times$ in $B^\times$.

Any admissible order (in the case $n\leq c$ or $K$ split) is an Eichler order of
discriminant $n$, i.e. conjugate to $R_0(n) := \begin{pmatrix}
  \CO & \CO \\
  \fp^n & \CO
\end{pmatrix}$. Choose an embedding of $K$ into $M_2(F)$ as follows such that $R_0(n)$ is an admissible order for $(\pi, \chi)$.
\begin{enumerate}
  \item If $K$ is split,  fix an $F$-algebra isomorphism $K \cong F^2$.
  If $c \geq n$ or $n(\chi_1) = c$, embed $K$
  into $M_2(F)$ by
\[\iota_1: (a, b) \lto\gamma_c^{-1}
  \begin{pmatrix}
    a & \\
      & b
  \end{pmatrix}\gamma_c, \quad \gamma_c =
  \begin{pmatrix}
    1 & \varpi^{-c} \\
      & 1
  \end{pmatrix}.\]
If $n(\chi_1) < c < n$, embed $K$ into $M_2(F)$ by
\[\iota_2: (a, b) \lto \gamma_c^{-1}
  \begin{pmatrix}
    b & \\
      & a
  \end{pmatrix}\gamma_c.\]
  Note that for any $t \in K^\times$, $\iota_1(t) = j\iota_2(t)j^{-1}$ with $j = \gamma_c^{-1}w\gamma_c$
  and $w =
  \begin{pmatrix}
    0 & 1 \\
    1 & 0
  \end{pmatrix}$.
 \item If $K$ is  a field, take $\tau \in \CO_K$ such that $\CO_K = \CO[\tau]$ and that
   if $K/F$ is ramified then $\tau$ is a uniformizer. Embed $K$ into $M_2(F)$ by
$$a+b\tau \lto
\gamma_c^{-1}
  \begin{pmatrix}
    a+b\tr \tau & b \RN\tau \\
    -b & a
  \end{pmatrix}\gamma_c, \quad \gamma_c =
  \begin{pmatrix}
    \varpi^c N\tau & \\
      & 1
  \end{pmatrix}.$$
\end{enumerate}


Assume $K \cong F^2$.
Note that if $n(\chi_1) < c < n$,
\[\beta^0 = \int_{F^\times \bs K^\times} \frac{(\pi(\iota_2(t))W_0,W_0)}{(W_0,W_0)}\chi(t)dt
=\int_{F^\times \bs K^\times} \frac{(\pi(\iota_1(t))W_0,W_0)}{(W_0,W_0)}\bar{\chi}(t)dt\]
where $\bar{\chi}_1 = \chi_2$, $\bar{\chi}_2 = \chi_1$ and $n(\bar{\chi}_1) = n(\chi_2) = c$.
We reduce to consider the case $c \geq n$ or
$n(\chi_1) = c$. Note that for the exceptional case, if we take $\pi(j)W_0$ as a test vector, then
\[\beta^0 = \int_{F^\times \bs K^\times} \frac{(\pi(\iota_1(t)j)W_0,\pi(j)W_0)}{(W_0,W_0)}\chi(t)dt
          = \int_{F^\times \bs K^\times} \frac{(\pi(\iota_1(t))W_0,W_0)}{(W_0,W_0)}\chi(\bar{t})dt
\]
with $n(\bar{\chi}_1) = n(\chi_2) = c$. Thus, even for the exceptional case, only need to consider
$W_0$ as a test vector. Thus
\[
  \begin{aligned}
  \beta^0 &= (W_0, W_0)^{-1}
  \iint_{(F^\times)^2} \pi(\gamma_c)W_0
  \left[
    \begin{pmatrix}
      ab & \\
         & 1
    \end{pmatrix}
  \right]
  \ov{\pi(\gamma_c)W_0
    \left[
      \begin{pmatrix}
	b & \\
	  & 1
      \end{pmatrix}
    \right]}\chi_1(a) d^\times b d^\times a \\
   &= (W_0,W_0)^{-1} |Z(1/2,\pi(\gamma_c)W_0,\chi_1)|^2.
  \end{aligned}
\]
If $c=0$,  $Z(1/2,W_0,\chi_1) = \chi_1(\delta)^{-1}L(1/2,\pi\otimes\chi_1)$
and then $\beta^0 =  (W_0,W_0)^{-1} L(1/2,\pi,\chi).$
If $c>0$, then
\[
  \begin{aligned}
    Z(1/2,\pi(\gamma_c)W_0,\chi_1) &= \int_{F^\times} W_0\left[
      \begin{pmatrix}
	a & \\
	  & 1
      \end{pmatrix}
    \right]\psi(a\varpi^{-c})\chi_1(a)d^\times a \\
    &= \sum_{k \in \BZ} W_0\left[
      \begin{pmatrix}
	\varpi^k & \\
	  & 1
      \end{pmatrix}
    \right]\int_{\varpi^k\CO^\times} \psi(a\varpi^{-c})\chi_1(a)d^\times a.
\end{aligned}\]
Assume $n(\chi_1) =c$, then the integral $\int_{\varpi^k\CO^\times} \psi(a\varpi^{-c})\chi_1(a)d^\times a$
vanishes unless $k = -v(\delta)$ while
\[\Big|\int_{\delta^{-1}\CO^\times} \psi(a\varpi^{-c})\chi_1(a)d^\times a\Big|= L(1,1_F)|\delta|^{1/2}q^{-c/2}.\]
Thus
\[\beta^0 = (W_0,W_0)^{-1}L(1,1_F)^2q^{-c}.\]
Assume $c\geq n$ and $n(\chi_1) < c$. Let $j$ be a normalizer of $K^\times$ with $jt = \bar{t}j$ for any
$t \in K^\times$. As $c \geq n$, there exists some $t_0 \in K^\times$ such that
$t_0U_0(n)t_0^{-1} = jU_0(n)j^{-1}$ and $\pi(t_0)W_0, \pi(j)W_0$ are in the same line. Thus
\[
  \begin{aligned}
  \beta^0 &= \int_{F^\times \bs K^\times}
  \frac{(\pi(t)W_0,W_0)}{(W_0,W_0)} \bar{\chi}(t)dt = (W_0,W_0)^{-1} L(1,1_F)^2q^{-c}
  \end{aligned}
\]
as $n(\bar{\chi}_1) = n(\chi_2) = c$.

\s{\bf Remark.} Assume $n(\chi_1)<c<n$ and $R$ is the intersection of two maximal orders $R'$ and
$R''$ with $R' \cap K = \CO_c$ and $R' \cap K = \CO_K$. If $R$ is not admissible, then the toric
integral for $f$ is $\omega$-eigen under $R^\times$ must vanish if $c > 1$. In the case $c=1$ and then $n(\chi_1) = 0$,
\[
  \int_{F^\times \bs K^\times}
  \frac{(\pi(\iota_1(t))W_0,W_0)}{(W_0,W_0)}\chi(t)dt = (W_0,W_0)^{-1}L(1,1_F)^2q^{-2}.
\]

It remains to consider the case $K$ is a field and $c \geq n$. Let $\Psi(g)$ denote the matrix coefficient:
    $$ \Psi(g) := \frac{(\pi(g)W_0,W_0)}{(W_0,W_0)}, \qquad  g\in \GL_2(F).$$
     Then
  $$\beta^0
    =\frac{\Vol(K^\times/F^\times)}{\# K^\times/F^\times\CO_c^\times}
  \sum_{t \in K^\times/F^\times\CO_c^\times} \Psi(t)\chi(t).$$
In this case $c=0$, $\pi$ is unramified.
    Further, if $K/F$ is unramified, then $\beta^0 = \Vol(K^\times/F^\times) = |\delta|^{1/2}$, and if $K/F$ is ramified,
    $\beta^0 = |D\delta|^{1/2} (1 + \Psi(\tau)\chi(\tau))$,
    where $\Psi(\tau)$ is expressed by the MacDonald polynomial and one has
    $\beta(f) = |D\delta|^{1/2}.$ It remains to consider the case $c>0$. Denote
    \[S_i = \left\{ 1+b\tau, b \in \CO/\fp^c, v(b) = i \right\}, \quad 0 \leq i \leq c-1\]
    and
    \[S' =
      \begin{cases}
    \left\{ a+\tau,a\in\fp/\fp^c \right\}, \quad &\text{if } e=1; \\
    \left\{ a\varpi+\tau, a \in \CO/\fp^c \right\}, \quad &\text{if } e=2.
      \end{cases}\]
    Then a complete representatives of $K^\times/F^\times\CO_c^\times$ can be taken as
    \[\left\{ 1 \right\}\sqcup\left(\sqcup_{i} S_i\right)\sqcup S'.\]
    Note that $\Psi$ is a function on $U_1(n)\bs G / U_1(n)$.
    The following observation is key to our computation:
    the images of $S_i$, $0 \leq i \leq c-1$ and $S'$ under the natural map
      \[\pr:K^\times \ra U_1(n)\bs G /U_1(n)\]
      are constant. Precisely,
      \[\pr(S_i) = \left[
    \begin{pmatrix}
      1 & \varpi^{i-c} \\
        & 1
    \end{pmatrix}
      \right], \quad
      \pr(S') = \left[
    \begin{pmatrix}
      & \varpi^{-c} \\
      -\varpi^{c+e-1} &
    \end{pmatrix}
      \right].\]
    Follow from this
    \[\sum_{t \in K^\times/F^\times\CO_c^\times} \Psi(t)\chi(t) =
    1+\sum_{i=0}^{c-1} \Psi_i \sum_{t\in S_i} \chi(t) + \Psi'\sum_{t \in S'} \chi(t),\]
    where $\Psi_i$ (resp. $\Psi'$) are the valuations of $\Psi(t)$ on $S_i$ (resp. $S'$).

    Assume the central character $\omega$ is unramified, then we may  take $\omega = 1$.
    If $e=c=1$, we have
      \[\sum_{t \in S_0} \chi(t) = -\chi(\tau)-1, \quad \sum_{t \in S'}\chi(t) = \chi(\tau).\]
     Otherwise
      \[\sum_{t \in S_i} \chi(t) =
    \begin{cases}
      0, \quad &\text{if $c>1$ and } 0 \leq i \leq c-2, \\
      -1,\quad &\text{if } i=c-1,
      \end{cases} \quad
      \text{and} \quad
      \sum_{t \in S'} \chi(t) = 0.\]
    Therefore,
    \[ \sum_{t \in K^\times/F^\times\CO_c^\times} \Psi(t)\chi(t) =
      \begin{cases}
    1+(-\chi(\tau)-1)\Psi_0 +\chi(\tau)\Psi', \quad &\text{if } e=c=1, \\
    1-\Psi_{c-1}, \quad &\text{otherwise.}
      \end{cases}\]
    Note that if $e=1$, then $
    \begin{pmatrix}
      & \varpi^{-c} \\
      -\varpi^{c} &
    \end{pmatrix}$ equals
    $
    \begin{pmatrix}
      1 & \varpi^{-c} \\
        & 1
    \end{pmatrix}$ in $ZU_1(n) \bs G / U_1(n)$ and as $\omega=1$, $\Psi' = \Psi_0$. We obtain
    \[\sum_{t \in K^\times/F^\times\CO_c^\times} \Psi(t)\chi(t) = 1-\Psi_{c-1}\]
    and reduce to evaluate $\Psi_{c-1}$. If $n=0$, the matrix coefficient $\Psi_{c-1}$ is
    expressed by the MacDonald polynomial.
    In particular, if  the Satake parameter of $\pi$ is $(\alpha,\alpha^{-1})$, then
    \[1-\Psi_{c-1} = \frac{(1-\alpha^2q^{-1})(1-\alpha^{-2}q^{-1})}{1+q^{-1}}.\]
    If $n=1$, then $\pi = \mathrm{sp}(2)\otimes \mu$ with $\mu$ a unramified quadratic character on
    $F^\times$. By definition,
    \[
      \begin{aligned}
    \Psi_{c-1} &= |\delta|^{1/2}L(1,\pi,\ad)^{-1}\int_{F^\times}
    W_0\left[
      \begin{pmatrix}
        a & \\
          & 1
      \end{pmatrix}
      \begin{pmatrix}
        1 & \varpi^{-1} \\
          & 1
      \end{pmatrix}
    \right]\ov{W_0\left[
      \begin{pmatrix}
        a & \\
          & 1
      \end{pmatrix}
    \right]}d^\times a  \\
    &= |\delta|^{3/2}L(1,\pi,\ad)^{-1}\int_{\varpi^{-n(\psi)}\CO}\psi(a\varpi^{-1})|a|^2d^\times a \\
    &= |\delta|^{3/2}L(1,\pi,\ad)^{-1}(-q^{-1})L(1,\pi,\ad)|\delta|^{-3/2} = -q^{-1}.
    \end{aligned}\]
    If $n \geq 2$, then
    \[
     \begin{aligned}
        \Psi_{c-1} &= |\delta|^{-1/2}\int_{\varpi^{-1-n(\psi)}\CO^\times} \psi(x)d^\times x =-q^{-1}L(1, 1_F).
     \end{aligned}
    \]
     With these results, we obtain
      \[\beta^0 = \frac{\Vol(K^\times/F^\times)}{\# K^\times/F^\times\CO_c^\times}\cdot
    \begin{cases}
      \displaystyle{\frac{L(1,1_F)}{L(1,\pi,\ad)(1+q^{-1})}}, \quad
      &\text{if } n=0; \\
      1+q^{-1}, \quad &\text{if }n=1; \\
      L(1,1_F), \quad &\text{if } n\geq 2.
      \end{cases}\]

  Finally, we deal with the case $\omega$ is ramified. As above, it is routine to check that $\Psi_i$ with
  $i < c-1$ and $\Psi'$ are vanishing. Moreover, $\Psi_{c-1} = 0$ if and only if $\delta_\pi = 0$ and
  for $\delta_\pi = 1$,
  \[\Psi_{c-1} = -q^{-1}L(1, 1_F).\]
  By the definition of $\delta_\pi$, if $\delta_\pi = 1$ then $c \geq 2$ and $n(\omega) < n \leq c$. Thus,
  for $\delta_\pi = 1$
    \[
      \begin{aligned}
      0 &= \sum_{t \in 1+\varpi^{c-1}\CO_K/1+\varpi^c\CO_K} \chi(t)\\
        &= \sum_{t \in 1+\varpi^{c-1}\CO_K/(1+\varpi^{c-1}\CO)(1+\varpi^c \CO_K)}
	\chi(t)\sum_{a \in 1+\varpi^{c-1}\CO/ 1+\varpi^c\CO} \omega^{-1}(a) \\
        &= q \sum_{b \in \fp^{c-1}/\fp^c} \chi(1+b\tau),
      \end{aligned}
      \]
  Therefore, if $\delta_\pi = 1$, then $\sum_{t \in S_{c-1}} \chi(t) = -1$  and
  \[\beta^0 = \frac{\Vol(K^\times/F^\times)}{\# K^\times/F^\times\CO_c^\times}\cdot
    \begin{cases}
      1, \quad &\text{if $\delta_\pi = 0$}; \\
      L(1,1_F), \quad &\text{if $\delta_\pi = 1$}.
  \end{cases}\]
  The proof of Proposition \ref{p-adic-beta} is now complete.

  We finish our discussions on $\alpha(W_0)$, $\beta(f)$ and $\gamma$ by the following Lemmas \ref{p-adic} and \ref{archimedean}.
  \begin{lem}\label{p-adic}
    Let $F$ be nonarchimedean and $f$ a nonzero element  in $V(\pi,\chi)$ then
    \[\alpha(W_0)\beta(f)\gamma|D|^{-1/2} = 2^{\delta(\Sigma_D)}L(1/2,\pi,\chi)^{-\delta(\Sigma)}
    L(1,\eta)^{2\delta(c_1)}q^{-c_1},\]
    where these $\delta\in \{0,1\}$ defined by
    \begin{itemize}
      \item $\delta(\Sigma_D) = 1$ if and only if $K$ is ramified, $n>0$ and $c<n$;
      \item $\delta(\Sigma) = 1$ if and only if $n>0$, $K$ is either ramified or $c>0$ and
	if $n=1$, then $c \geq n$;
      \item $\delta(c_1) = 1$ if and only if $c_1 \not=0$.
    \end{itemize}
 \end{lem}
 \begin{proof}
   We have computed $\alpha(W_0)$ in Proposition \ref{local inner product} and
   $\beta(f)$ in Proposition \ref{p-adic-beta}. When $n>0$, by Lemma \ref{gamma}, $\gamma = L(1,1_F)
   (1-e(R)q^{-1})$ and we reduce to compute $e(R)$:
   \begin{enumerate}
  \item[(i)] $e(R) = 1$ and $\gamma = 1$ if $K$ is split, or $K$ is ramified, $n=1$ and $B$ is split, or
    $K$ is nonsplit and $c \geq n$;
  \item[(ii)] $e(R) = -1$ and $\gamma = L(1,1_F)(1+q^{-1})$
   if $K$ is inert and $c<n$, or $K$ is ramified, $n=1$, $B$ is division and
    $c=0$;
  \item[(iii)] $e(R) = 0$ and $\gamma = L(1,1_F)$ if $K$ is ramified, $n \geq 2$ and $c < n$.
  \end{enumerate}

 \end{proof}

 For  archimedean places,
 using Barnes' lemma, we have the following list for $( W_0, W_0)$ (see \cite{Tadic} for
 the classification of unitary dual of $\GL_2(F)$):
 \begin{enumerate}
  \item Assume $F = \BR$, $\sigma$ is the infinite dimensional subquotient of the induced representation
    $\Ind(\mu_1,\mu_2)$ where $\mu_i(a) = |a|^{s_i}\sgn(a)^{m_i}$ with $s_i \in \BC$ and
    $m_i \in \{0,1\}$. Denote $k = s_1 - s_2 + 1$, $\mu = s_1 + s_2$.
    \begin{enumerate}
      \item If $\sigma = \CD_\mu(k)$ is the discrete series with $k \geq 2$, then $(W_0,W_0)$ equals
	\[2(4\pi)^{-k}\Gamma(k).\]
      \item If $\sigma = \pi(\mu_1,\mu_2)$  is a principal series, then $(W_0,W_0)$ equals
	\[\pi^{-1-m_1 - m_2}\Gamma(\frac{1+2m_1}{2})
	\Gamma(\frac{1+2m_2}{2}) B( \frac{k+m_1+m_2}{2}, \frac{2-k + m_1 + m_2}{2}),\]
      where $B(x, y):=\Gamma(x)\Gamma(y) \Gamma(x+y)^{-1}$ is the beta function.
      \end{enumerate}
   \item Assume $F = \BC$, $\sigma = \pi(\mu_1,\mu_2)$ is a principal series
     with $\mu_i(z) = |z|^{s_i}(z/\sqrt{|z|_\BC})^{m_i}$ and $s_i \in \BC$ and $m_i \in \BZ$, then
     $(W_0,W_0)$ equals
   	 \[\quad\quad\quad 8 (2\pi)^{-1-|m_1|-|m_2|}\Gamma(1+|m_1|)
	   \Gamma(1+|m_2|)B(1+s_1-s_2 +\frac{|m_1|+|m_2|}{2},\
	 1-s_1+ s_2 + \frac{|m_1| + |m_2|}{2}).\]
\end{enumerate}
For a pair $(\pi,\chi)$, define
\[C(\pi,\chi) =
   \begin{cases}
     2^{-1}\pi(W_0,W_0)^{-1}, \quad &\text{if $K/F = \BC/\BR$}; \\
     (W_0',W_0')(W_0,W_0)^{-1}, \quad &\text{if $K = F^2$}.
 \end{cases}\]
 In the split case, $W_0'$ is the new vector of $\pi\otimes\chi_1$ where $K$ is embedded  into
 $M_2(F)$ diagonally and $\chi_1(a)=\chi\left(\matrixx{a}{}{}{1}\right)$.

 \begin{lem}\label{archimedean} For $F$ archimedean,
 let $f$ be a nonzero vector in $V(\pi,\chi)$,  then
 \[ \alpha(W_0)\beta(f) =
   C(\pi,\chi)^{-1}\begin{cases}
     L(1/2,\pi,\chi)^{-1}, \quad &\text{if $K/F = \BC/\BR$};\\
     1, \quad &\text{if $K = F^2$}.
 \end{cases}\]
 In particular, if
 $\sigma = \CD_\mu(k)$ is a discrete series with weight $k$, then
 \[C(\pi,\chi) =
   \begin{cases}
     4^{k-1}\pi^{k+1}\Gamma(k)^{-1}, \quad &\text{if $K = \BC$}; \\
     1, \quad &\text{if $K = \BR^2$}.
 \end{cases}\]
\end{lem}
\begin{proof}
  By definition,
  \[\alpha(W_0)\beta(f) = \frac{L(1,\eta)}{L(1,1_F)}L(1/2,\pi,\chi)^{-1}(W_0,W_0)
  \beta^0,\]
  with

  \[\beta^0 = \int_{F^\times \bs K^\times} \frac{(\pi(t)f,f)}{(f,f)}
\chi(t)dt, \quad f \in V(\pi,\chi).\]
  If $K/F = \BC/\BR$, then $\beta^0 = \Vol(K^\times/F^\times) = 2$.  If $K$ is split,  taking $f=W_0'$,  then
  $\beta^0 = L(1/2,\pi,\chi)(W_0',W_0')^{-1}$. If $\sigma = \CD_\mu(k)$, the value
  for $(W_0,W_0)$ is in the above list (1a) and we note that if $K = \BR^2$, then
  $(W_0',W_0') = (W_0,W_0)$ as for any $\chi_1$, $\pi\otimes\chi_1$ and $\pi$
  have the same weight.
\end{proof}

\end{document}